\documentclass[11pt,a4paper,leqno]{amsart}
\usepackage{a4wide} 
\usepackage{latexsym}
\usepackage{amsmath}
\usepackage{amssymb} 
\usepackage{stackrel} 
\usepackage{color}
\usepackage{graphicx}
\usepackage{hyperref}
\usepackage{amsthm,amstext,units,enumerate}

\newtheorem{lemma}{Lemma}
\newtheorem{prop}{Proposition}
\newtheorem{theo}{Theorem}
\newtheorem{corol}{Corollary}
\theoremstyle{definition}
\newtheorem{defin}{Definition}

\newcommand{\C}{\mathbb{C}}
\newcommand{\N}{\mathbb{N}}
\newcommand{\R}{\mathbb{R}}

\newcommand{\Oo}{\mathcal{O}}
\newcommand{\Bo}{\mathcal{B}}
\newcommand{\Rr}{\mathcal{R}}
\newcommand{\EE}{\mathbb{E}}
\newcommand{\M}{\mathbb{M}}

\begin{document}
\title[Multisummability singularly perturbed moment differential equations]{Multisummability of formal solutions for a family of generalized singularly perturbed moment differential equations}
\author{Alberto Lastra}
\address{Departamento de F\'isica y Matem\'aticas\\
University of Alcal\'a\\
Ap. de Correos 20, E-28871 Alcal\'a de Henares (Madrid), Spain}
\email{alberto.lastra@uah.es}
\author{S{\l}awomir Michalik}
\address{Faculty of Mathematics and Natural Sciences,
College of Science\\
Cardinal Stefan Wyszy\'nski University\\
W\'oycickiego 1/3,
01-938 Warszawa, Poland}
\email{s.michalik@uksw.edu.pl}
\urladdr{\url{http://www.impan.pl/~slawek}}
\author{Maria Suwi\'nska}
\address{Faculty of Mathematics and Natural Sciences,
College of Science\\
Cardinal Stefan Wyszy\'nski University\\
W\'oycickiego 1/3,
01-938 Warszawa, Poland}
\email{m.suwinska@op.pl}
\date{}
\begin{abstract}

The notion of moment differentiation is extended to the set of generalized multisums of formal power series via an appropriate integral representation and accurate estimates of the moment derivatives. 

The main result is applied to characterize generalized multisummability of the formal solution to a family of singularly perturbed moment differential equations in the complex domain, broadening widely the range of singularly perturbed functional equations to be considered in practice, such as singularly perturbed differential equations and singularly perturbed fractional differential equations.
   
\end{abstract}

\maketitle
\thispagestyle{empty}

\section{Introduction}

The main aim of this work is to give a step forward in the theory of summability of formal solutions to functional equations in the complex domain. More precisely, we deal with the so-called moment (partial) differential equations, in which the operators known as moment derivatives act on the unknown function. 

The main advances in the present study are twofold: on the one hand, we provide an integral representation of the moment derivatives of generalized sums of formal power series which can be extended to an infinite sector satisfying some generalized exponential growth. In addition to this, we describe the dependence of such moment derivatives with respect to three elements, namely the moment sequence, the sequence involved in the asymptotic representation of the sum, and also the generalized exponential growth at infinity (Theorem~\ref{teo2}). Consequently, a novel definition of moment derivatives acting on generalized multisums of formal power series makes sense, and induces many possible applications in the theory, as the following one. On the other hand, we apply the previous result to achieve the main result of this research, namely the multisummability of the formal solutions to certain family of singularly perturbed moment differential equations (Theorem~\ref{teopral}). More precisely, we prove that the formal solution to
\begin{equation}\label{eq:main:intro}
 \left\{
 \begin{aligned}
 \varepsilon^k a(z)\partial_{m_2,z}^p\omega(z,\varepsilon)-\omega(z,\varepsilon)&=\hat{f}(z,\varepsilon)\\
  \partial_{m_2,z}^j\omega(0,\varepsilon)&=\hat{\psi}_j(\varepsilon),\qquad j=0,\ldots, p-1,
 \end{aligned}
 \right.
\end{equation}
is multisummable along a certain appropriate multidirection $(d_1,d_2)\in\R^2$ with respect to the perturbation parameter $\varepsilon$ if and only if the forcing term $\hat{f}(z,\varepsilon)\in\C[[z,\varepsilon]]$ and the initial conditions $\hat{\psi}_{j}(\varepsilon)\in\C[[\varepsilon]]$, for $0\le j\le p-1$ are multisummable along the same multidirection. Here, $\varepsilon$ stands for a small complex parameter, $a(z)$ is a holomorphic function near the origin, $m_2$ is a sequence of moments, and $k,p$ are positive integers with $k<p$. The precise description of the elements involved in the problem is given in Section~\ref{secfinal}.

The study of moment differential equations is motivated by the versatility of the moment derivative operator, whose definition was initially put forward by W. Balser and M. Yoshino in~\cite{BY} for formal power series. Let $m=(m(p))_{p\ge0}$ be a sequence of positive real numbers. Then, the moment derivative $\partial_{m,z}:\C[[z]]\to\C[[z]]$ is defined by $\partial_{m,z}(z^{p})=\frac{m(p)}{m(p-1)}z^{p-1}$ for every positive integer $p$ and $\partial_{m,z}(1)=0$, defining moment derivatives for elements in $\C[[z]]$ by linearity. It is natural to extend the previous definition to a holomorphic function on some disc at the origin by identifying the function with its Taylor series at $z=0$. In addition to this, the authors proved that one can extend the definition of moment differentiation to the generalized sums in a direction of a formal power series, as the sum of the moment derivative of that formal power series (see Corollary 1 and Definition 10,~\cite{LMS2}). In this work, we further extend it to the generalized multisums of a formal  power series (Corollary~\ref{coro519} and Definition~\ref{defi624}).

The most suitable choice for $m$ is to be a moment sequence of certain Laplace-like operator. It is clear that the sequence $m_1=(p!)_{p\ge0}$ gives rise to the usual derivative when considering the moment operator $\partial_{m_1,z}$. Apart from that choice for $m$, many other derivations which appear in concrete applications can be represented as a moment derivative. For example, for every $k>0$, the sequence $m_{1/k}=(\Gamma(1+p/k))_{p\ge0}$ is associated with the Caputo fractional derivative $\partial_z^{1/k}$ via the relation $(\partial_{m_{1/k},z}\hat{f})(z^{1/k})=\partial_{z}^{1/k}(\hat{f}(z^{1/k}))$, valid for every $\hat{f}(z)\in\C[[z]]$ (see~\cite{michalik13}, Remark 3, for further details). Recently, many applications of Caputo derivatives appear in the literature such as~\cite{gomoyunov,KST}, also in the study of the asymptotic periodic solutions of evolution equations~\cite{RWF}, numerical studies, etc. Let $q\in(0,1)$. The sequence $m_q=([p]_q!)_{p\ge0}$, with $[p]!_q=[1]_q[2]_q\cdots[p]_q$ is known as the sequence of $q-$factorials,  where $[\ell]_q=\sum_{j=0}^{\ell-1}q^{j}$. It determines the moment differentiation which coincides with the $q-$derivative $D_{q,z}$ given by $D_{q,z}z^p=[p]_{q}z^{p-1}$ for every $p\in\N$. This moment differentiation is quite related to the dilation operator, appearing in the study of $q-$difference equations which is of great interest in the scientific community with interesting advances in the knowledge of the asymptotic behavior of the solutions of $q-$difference equations (see~\cite{lama,malek} among others, and the references therein). 

The interest of moment functional equations has increased in the last decade, and recent achievements have been reached in this concern. A first step was given in the seminal work~\cite{BY}, where the authors study the formal solutions and Gevrey estimates of their coefficients of linear moment partial differential equations with constant coefficients. The development of a more general theory through the construction of kernels for generalized summability by J. Sanz in~\cite{sanz} allows to enlarge the class of moment sequences considered, in the framework of strongly regular sequences. The sequence $m_{1/k}$ above belongs to such a family for every $k>0$, whereas $m_q$ does not. The so-called $1+$ level, appearing in the asymptotic study of difference equations is also related to a strongly regular sequence~\cite{i,i2}. 

After the seminal work~\cite{BY}, the second author gave answer to the problem of analyticity of such problems~\cite{michalik13}, via splitting of the characteristic equation with respect to one of its variables. We also refer to~\cite{m2} for a further study in the homogeneous situation, while dropping the condition of the convergence of the initial data. Further knowledge on the solutions to moment partial differential equations with constant coefficients is given in~\cite{mt} to study the growth properties and summability of the formal solutions. It is also worth mentioning the family of partial differential equations studied in~\cite{lastramaleksanz}, where the coefficients of the equation under study belong to certain functional spaces associated with functions whose derivatives are uniformly bounded in terms of some strongly regular sequence.

The last advances in the theory of the asymptotic behavior of solutions to moment functional equations have been obtained recently regarding the summability of certain families of moment integro-differential equations~\cite{LMS}, and also Maillet-type theorems~\cite{LMS0,su}. We mention the recent works by P. Remy in the study of partial differential equations~\cite{re} and integro-differential equations~\cite{re2} of a similar nature as those considered in these works.

Recent results on generalized multisummability of formal power series concerning different (comparable and nonequivalent) levels associated with ultraholomorphic classes achieved in~\cite{jkls} are applied in the present study to achieve asymptotic properties of the solutions to a singularly perturbed moment differential equation. That concept of multisummability as long as previous results achieved by the authors in~\cite{LMS2} have been the key points used to describe generalized multisummability of the formal solution of the main equation (\ref{eq:main:intro}). As mentioned above, novel integral representations and accurate estimates of the moment derivatives of generalized sums of formal power series are needed, arriving to the coherent definition of moment derivation of the generalized multisum of a formal power series. More precisely, the first main result of the present study, Theorem~\ref{teo2}, resorts to an appropriate deformation path which is split in order to provide upper bounds for the moment derivative of a sectorial holomorphic function, quite related to the multisummability process. As a consequence, Corollary~\ref{coro519} and Definition~\ref{defi624} encompass the notion of moment derivation of the multisum of a formal power series along some multidirection. As an application of these results, Section~\ref{secfinal} characterizes multisummability of the formal solution to (\ref{eq:main:intro}) in terms of that of the initial data and the forcing term (Theorem~\ref{teopral}). The proof of this last result is based on the properties of formal Borel transform which transform the problem into an auxiliary moment partial differential equation, which is easier to handle.

The paper is structured as follows. We fix notation in Section~\ref{secnotacion}, followed by Section~\ref{secprel}, where we recall the definition and main results on strongly regular sequences, generalized summability and multisummability of formal power series. In that section, some technical results needed in the sequel are also proved. The main purpose of Section~\ref{secmomdif} is to state Theorem~\ref{teo2}, leading to a coherent definition of moment derivative of the generalized multisum of a formal power series. Section~\ref{secfinal} describes an application (Theorem~\ref{teopral}) of the previous results in the framework of singularly perturbed moment differential equations.

\section{Notation}\label{secnotacion}
By $\N$ we shall denote the set of all positive integers, i.e., $\{1,2,\dots\}$ and $\N_0=\N\cup\{0\}$.

$\Rr$ stands for the Riemann surface of the logarithm.

For all $r>0$ and $z_0\in\C$, $D(z_0,r)$ stands for the open disc in the complex plane centered at $z_0$ and with radius $r$. For any fixed $\theta>0$ and $d\in\R$ a subset of $\Rr$ defined as
$$
S_d(\theta)=\left\{z\in\Rr:\ |\arg z-d|<\frac{\theta}{2}\right\}
$$
is an open infinite sector with vertex at the origin, bisecting direction $d$ and opening $\theta$. In cases where the opening is not specified, we simply write $S_d$. For every $r>0$, we write $S_{d}(\theta;r):=S_{d}(\theta)\cap D(0,r)$. A sectorial region $G_d(\theta)$ is a subset of $\Rr$ such that there exists $r>0$ for which $G_d(\theta)\subset S_d(\theta;r)$ and for any $0<\theta'<\theta$ there exists $0<r'<r$ such that $S_d(\theta';r')\subset G_d(\theta)$.

We put $\hat{S}_{d}(\theta;r):=S_{d}(\theta)\cup D(0,r)$. Analogously, we write $\hat{S}_{d}(\theta)$ (resp. $\hat{S}_{d}$) whenever the radius $r>0$ (resp. the radius and the opening $r,\theta>0$) can be omitted. We write $S\prec S_d(\theta)$ if $S$ is an infinite sector with vertex at the origin such that $\bar{S}\subset S_d(\theta)$, where the closure is considered with respect to $\mathcal{R}$. Similarly, for two sectorial regions $G_d(\theta)$ and $G_{d'}(\theta')$ we write $G_d(\theta)\prec G_{d'}(\theta')$ whenever $G_d(\theta)\subset G_{d'}(\theta')$ and relation $\prec$ holds for the sectors appearing in the definitions of both sectorial regions.

If $(\EE,\|\cdot\|_\EE)$ is a complex Banach space, by $\Oo(U,\EE)$ we denote the set of all functions holomorphic on the open set $U\subset \C$ with values from $\EE$. For $\EE=\C$ we simply write $\Oo(U)$. The set of all formal power series in $t$ with coefficients in $\EE$ is denoted by $\EE[[z]]$.

\section{Preliminary results and definitions}\label{secprel}

The purpose of this section is to recall the main facts on the theory of generalized summability and also one of the equivalent notions of generalized multisummability of a formal power series along certain multidirection, developed in~\cite{jkls}. We first remind the main elements regarding the theory of strongly regular sequences and related properties together with the theory of generalized summability, for the sake of completeness. These definitions and the detailed constructions can be found in~\cite{sanzproceedings} and the references therein.

\subsection{Strongly regular sequences}

The concept of strongly regular sequences was put forward by V.~Thilliez in~\cite{thilliez}.

\begin{defin}
Let $\mathbb{M}=(M_p)_{p\ge0}$ be a sequence of positive real numbers with $M_0=1$. $\mathbb{M}$ is a \emph{strongly regular sequence} if the following statements hold: 
\begin{itemize}
\item[(lc)] $M_p^2\le M_{p-1}M_{p+1}$, $p\ge1$ ($\mathbb{M}$ is \emph{logarithmically convex}).
\item[(mg)] there exists $A_1>0$ such that $M_{p+q}\le A_1^{p+q}M_pM_q$, for all $p,q\ge0$ ($\mathbb{M}$ is of \emph{moderate growth}).
\item[(snq)] there exists $A_2>0$ such that $\sum_{q\ge p}\frac{M_q}{(q+1)M_{q+1}}\le A_2\frac{M_p}{M_{p+1}}$, for all $p\ge0$ ($\mathbb{M}$ satisfies the \emph{strong non-quasianalyticity condition}). 
\end{itemize}
\end{defin}

The previous notion generalizes that of Gevrey sequences of order $\alpha>0$, $(p!^{\alpha})_{p\ge0}$, which widely appear in the theory of summability of formal solutions to functional equations. In association with a strongly regular sequence $\mathbb{M}$, one can define the function 
\begin{equation}\label{eq:function_M}
 M(t):=\left\{
 \begin{aligned}
 \sup_{p\ge 0}\log \left(\frac{t^p}{M_p}\right)&\quad\textrm{ for }t>0\\
 \\0\qquad\quad&\quad \textrm{ for }t=0
 \end{aligned}
 \right.
\end{equation}
It turns out that $M$ is a non-decreasing and continuous function in $[0,\infty)$, with $\lim_{t\to\infty}M(t)=+\infty$. We also consider the positive real number
$$\omega(\mathbb{M}):=\left(\lim\sup_{r\to\infty}\max\left\{0,\frac{\log(M(r))}{\log(r)}\right\}\right)^{-1},$$
which determines the limit opening for a sector to admit nontrivial flat ultraholomorphic functions defined on them. We refer to~\cite{jss} for a deeper study in this direction. 

Following~\cite{sanz,sanzproceedings,thilliez}, one has the next results.

\begin{lemma}[(17),~\cite{sanz}]\label{lema_1}
For every $H>0$, there exist $C,D>0$ such that for all $p\ge0$ one has
$$\int_0^{\infty}t^{p-1}\exp(-M(t/H))dt\le CD^pM_p.$$
\end{lemma}

\begin{lemma}\label{lema0}
Let $\mathbb{M}$ be a strongly regular sequence, and let $s\ge1$. There exists $\rho(s)\ge 1$ (only depending on $\mathbb{M}$ and $s$) such that 
$$\exp(-M(t))\le \exp(-sM(t/\rho(s))),\qquad t\ge 0.$$ 
\end{lemma}

\begin{lemma}\label{lema1}
Let $\mathbb{M}=(M_p)_{p\ge0}$ be a strongly regular sequence. The sequence $\mathbb{M}^s:=(M_p^{s})_{p\ge0}$ defines a strongly regular sequence for every $s>0$. Moreover, $\omega(\mathbb{M}^s)=s\omega(\mathbb{M})$.
\end{lemma}

\begin{lemma}\label{lema2}
Let $\mathbb{M}$ be a sequence satisfying (lc) property. Then 
\begin{itemize}
\item $(M_p^{1/p})_{p\ge0}$ is nondecreasing. 
\item $M_pM_q\le M_{p+q}$ for all $p,q\in\N_0$.
\end{itemize}
\end{lemma}

\subsection{Generalized summability}

In this subsection, $(\mathbb{E},\left\|\cdot\right\|_{\mathbb{E}})$ stands for a complex Banach space.

The classical summability theory of formal power series related to Gevrey sequences (see for example~\cite{balser,loday}) has recently been adapted to the more general settings involving strongly regular sequences (see~\cite{sanz,sanzproceedings}). This notion leans on the approximation of holomorphic functions in sectors of the complex plane by formal power series whenever the approximation is given in terms of a given strongly regular sequence.

\begin{defin}
Let $\mathbb{M}=(M_p)_{p\ge0}$ be a sequence of positive real numbers, and let $G_{d}(\theta)\subseteq\mathcal{R}$ be a sectorial region, for some $\theta>0$ and $d\in\R$. A function $f\in\mathcal{O}(G_d(\theta),\mathbb{E})$ admits the formal power series $\hat{f}(z)=\sum_{p\ge0}f_pz^p\in\mathbb{E}[[z]]$ as its \emph{$\mathbb{M}$-asymptotic expansion in $G_d(\theta)$} if for every $0<\theta'<\theta$ and $r>0$ with $S_d(\theta';r)\subseteq G_{d}(\theta)$ and all integer $N\ge1$, there exist $C,A>0$ with
$$\left\|f(z)-\sum_{p=0}^{N-1}f_pz^p\right\|_{\mathbb{E}}\le CA^{N}M_N|z|^{N},$$
for all $z\in S_{d}(\theta';r)$.
\end{defin}

Further details on the following result can be found in Section 3~\cite{sanzproceedings}.

\begin{lemma}\label{lema140}
In the situation of the previous definition, there exist $\tilde{C},\tilde{A}>0$ such that 
$$\left\|f_p\right\|_{\mathbb{E}}\le \tilde{C}\tilde{A}^pM_p$$
for every $p\ge0$.
\end{lemma}

The exponential growth in sectors of the complex plane is extended in terms of the function $M(\cdot)$ as follows. 

\begin{defin}\label{defi3}
 Let $\theta>0$ and $d\in \R$, and suppose that $\M$ is a fixed sequence of positive real numbers. We define the set $\Oo^{\M}(S_d(\theta),\EE)$ as consisting of all functions $f\in\Oo(S_d(\theta),\EE)$ such that for every $0<\theta'<\theta$ there exist constants $C,K>0$ satisfying
 \begin{equation}\label{e152}
  \|f(z)\|_{\EE}\le C\exp\left(M\left(\frac{|z|}{K}\right)\right) \textrm{ for every }z\in S_d(\theta').
 \end{equation}
\end{defin}


The construction of operators involved in the summability process leans on the existence of kernel functions for generalized summability, related to a given strongly regular sequence.

\begin{defin}\label{defin:kernel-functions}
 Let $\M$ be a strongly regular sequence with $\omega(\M)<2$ and with function $M(\cdot)$ defined as in \eqref{eq:function_M}. Two complex functions $e, E$ are \emph{strong kernel functions for $\M$-summability} if the following properties hold:
\begin{itemize}
\item $e\in\Oo(S_0(\omega(\M)\pi))$. There exists $\alpha>0$ such that for all bounded proper subsectors $T$ of $S_{0}(\omega(\mathbb{M}\pi))$, there exists $C>0$ with
\begin{equation}\label{eq:est_e}
|e(z)|\le C|z|^{\alpha},\qquad z\in T.
 \end{equation}

Furthermore, $e$ is a flat function in every infinite subsector of $S_0(\omega(\mathbb{M}))$. More precisely, for every $\varepsilon>0$ there exist $C, K>0$ such that
\begin{equation}\label{e162}
|e(z)|\le C\exp\left(-M\left(\frac{|z|}{K}\right)\right)\quad \hbox{ for every }z\in S_0(\omega(\M)\pi-\varepsilon).
\end{equation}
We also assume that $e(x)\in\R$ for every real $x>0$.
\item $E\in\Oo(\C)$ with generalized exponential growth at infinity
\begin{equation}\label{kernel2}
|E(z)|\le \tilde{c}\exp\left(M\left(\frac{|z|}{\tilde{k}}\right)\right)\quad\textrm{for every } z\in \C,
\end{equation}
for some $\tilde{c},\tilde{k}>0$. There also exists $\beta>0$ such that for all $0<\tilde{\theta}<2\pi-\omega(\M)\pi$ and $M_E>0$, there exists $\tilde{c}_2>0$ with 
\begin{equation}\label{e202}
|E(z)|\le \frac{\tilde{c}_2}{|z|^{\beta}},\quad z\in S_\pi(\tilde{\theta})\setminus D(0,M_E).
\end{equation}
\item Functions $e$ and $E$ are connected by the \textit{moment function associated with $e$} defined by
\begin{equation}\label{kernel3}
m_{e}(z):=\int_{0}^{\infty} t^{z-1}e(t)dt.
\end{equation}
$m_e$ is a holomorphic function on $\{z\in\C:\hbox{Re}(z)> 0\}$, continuous up to its boundary. Indeed, $E$ is determined from $e$ via the sequence of moments associated with $e$, $(m_e(p))_{p\ge0}$, by
\begin{equation}\label{kernel4}
E(z)=\sum_{p\ge0}\frac{z^p}{m_{e}(p)},\quad z\in\C.
\end{equation}
\end{itemize}
\end{defin}

\begin{lemma}[Proposition 5.7,~\cite{sanz}]\label{lema_2}
Given a kernel function $e$ for $\mathbb{M}$-summability, associated with some strongly regular sequence $\mathbb{M}$, the sequence of moments $(m_e(p))_{p\ge0}$ and $\mathbb{M}$ are equivalent, i.e., there exist $C,D,\tilde{C},\tilde{D}>0$ such that
$$C D^pm_e(p)\le M_p\le \tilde{C}\tilde{D}^pm_e(p),\qquad p\ge0.$$
\end{lemma}

As a matter of fact, the classical kernels for summability involved in the Gevrey theory satisfy weaker properties (see~\cite{balser}). These more restrictive conditions are justified (see Section 4.2,~\cite{jkls}) regarding their applicability and adaptability to practical situations. Indeed, given a strongly regular sequence, the existence of a pair of kernels for $\mathbb{M}$-summability is guaranteed whenever $\mathbb{M}$ admits a nonzero proximate order. We refer to Section 2.3~\cite{jkls} for a brief review on sequences admitting a nonzero proximate order. This property will turn into an assumption for every strongly regular sequence under consideration hereinafter.

The classical formal Borel transform can also be adapted to this framework.

\begin{defin}\label{defin:borel-formal}
 Let $(m_e(p))_{p\ge 0}$ be a sequence of moments. Then the \emph{formal $m_e$-Borel moment transform} $\hat{\Bo}_{m_e,z}\colon \EE[[z]]\to\EE[[z]]$ is given by
 $$
 \hat{\Bo}_{m_e,t}\left(\sum_{p\ge 0}u_p z^p\right)=\sum_{p\ge 0}\frac{u_p}{m_e(p)}z^p.
 $$
\end{defin}

An $\mathbb{M}$-analog of Laplace transform is also available (see Section 6,~\cite{sanzproceedings}).

\begin{prop}\label{prop1}
Let $d\in\R$. Given a strongly regular sequence $\mathbb{M}$ which admits a nonzero proximate order, and a pair of kernel functions for $\mathbb{M}$-summability associated, say $e$ and $E$, we define for every $f\in\mathcal{O}^{\mathbb{M}}(S_d,\mathbb{E})$ the $e$-Laplace transform of $f$ along direction $\tau\in\hbox{arg}(S_d)$ by
$$(T_{e,\tau}f)(z)=\int_0^{\infty(\tau)}e(u/z)f(u)\frac{du}{u},$$
for all $|\hbox{arg}(z)-\tau|<\omega(\mathbb{M})\pi/2$, and small enough $|z|$. The variation of $\tau$ among the arguments of $S_d$ determines a holomorphic function, denoted $T_{e,d}f$, defined in a sectorial region of bisecting direction $d$ and opening larger than $\omega(\mathbb{M})\pi$.
\end{prop}

As a matter of fact, there exists a generalization to the classical Borel-Laplace procedure for the effective summation of a given formal power series. 

\begin{defin}\label{def198}
Let $\M$ be a strongly regular sequence which admits a nonzero proximate order and let $m_e$ denote a sequence of moments associated with $\M$. The series $\hat{u}\in\mathbb{E}[[z]]$ is \emph{$\M$-summable along direction $d\in\R$} if $\hat{\Bo}_{m_e,z}(\hat{u}(z))$ is a series with a positive radius of convergence, and the analytic function defining such series, say $u(z)$, can be extended to an infinite sector of bisecting direction $d$, say $\hat{S}_d$, with $u(z)\in\Oo^{\M}(\hat{S}_d,\mathbb{E})$.
\end{defin}

\begin{prop}\label{prop2}
In the situation of the previous definition, the function $v(z)=(T_{e,d}u)(z)$ is holomorphic on a bounded sector of bisecting direction $d$ and opening larger than $\omega(\mathbb{M})\pi$. 
\end{prop}

Definition~\ref{def198} does not depend on the kernel functions for $\mathbb{M}$-summability (and therefore on the moment sequence) considered. In addition to this, the procedure described there provides us with the only function (due to Watson's Lemma, see Corollary 3.16,~\cite{jss}) admitting the initial formal power series as its $\mathbb{M}$-asymptotic expansion in a wide enough sector of bisecting direction $d$, known as the $\mathbb{M}$-sum of the formal power series along direction $d$.

\begin{defin}
The function $v$ in Proposition~\ref{prop2} is known as the \emph{$\mathbb{M}$-sum of $\hat{u}$ along direction $d\in\R$}, and is denoted by $\mathcal{S}_{\mathbb{M},d}(\hat{u})$. The set of formal power series with coefficients in $\mathbb{E}$ which are $\mathbb{M}$-summable along direction $d$ is denoted by $\mathbb{E}\{z\}_{\mathbb{M},d}$. 
\end{defin}

\begin{lemma}\label{lema100}
Let $\hat{f}(z)$ be a formal power series, and let $k\in\N_0$. We define the formal power series $\hat{g}(z):=z^k \hat{f}(z)$. 

Let $\mathbb{M}$ be a strongly regular sequence admitting a nonzero proximate order. Let $d\in\R$. Then, the following statements are equivalent:
\begin{itemize}
\item The formal power series $\hat{f}$ is $\mathbb{M}$-summable in the direction $d$.
\item The formal power series $\hat{g}$ is $\mathbb{M}$-summable in the direction $d$.
\end{itemize}
If one of the previous equivalent statements holds, then 
\begin{equation}\label{e282}
\mathcal{S}_{\mathbb{M},d}(\hat{g})=z^k\mathcal{S}_{\mathbb{M},d}(\hat{f}).
\end{equation}
\end{lemma}
\begin{proof}
It is straightforward that $\hat{f}$ being $\mathbb{M}$-summable in the direction $d$ yields $\hat{g}$ being $\mathbb{M}$-summable in the same direction, as the set of $\mathbb{M}$ summable functions in a direction is an algebra. The second part of the equivalence can be proved following an analogous argument as that for Exercise 3, Section 4.5, in~\cite{balser}. More precisely, an iterative argument allows us to assume that $k=1$. Let $\hat{g}(z)=\sum_{p\ge1}g_pz^p.$ Then, there exists a bounded sector with bisecting direction $d$ and opening larger than $\omega(\mathbb{M})\pi$ such that for any subsector $T$ there exist $C,A>0$ with  
$$\left\|\mathcal{S}_{\mathbb{M},d}(\hat{g})(z)-\sum_{p=1}^{n-1}g_pz^p\right\|\le C A^nM_n|z|^{n},$$
valid for every $z\in T$ and $n\ge2$. Therefore, one has that
$$\left\|z^{-1}\mathcal{S}_{\mathbb{M},d}(\hat{g})(z)-\sum_{p=0}^{n-2}g_{p+1}z^{p}\right\|\le C A^nM_n|z|^{n-1}\le CAA_1M_1(AA_1)^{n-1}M_{n-1}|z|^{n-1},$$
regarding property $(mg)$ of $\mathbb{M}$. Observe that $\hat{f}(z)=\sum_{p\ge0}g_{p+1}z^p$, which concludes the proof.
\end{proof}

\subsection{Generalized multisummability}

As in the classical theory, the procedure of Borel-Laplace summation does not succeed when dealing with the formal solutions to some functional equations. As a matter of fact, a more general approach called multisummability is needed in the study of formal solutions to ordinary differential equations. A generalized theory of multisummability can be considered in this framework from different points of view. 

The theory of generalized multisummability deals with summability processes with respect to sequences obtained by algebraic actions on the initial strongly regular sequences handled. More precisely, given two sequences of positive real numbers $\mathbb{M}=(M_p)_{p\ge0}$ and $\mathbb{L}=(L_p)_{p\ge0}$, we denote $\mathbb{M}/\mathbb{L}:=(M_p/L_p)_{p\ge0}$. The comparison of sequences and their properties is studied in Sections 3.1 and 3.2~\cite{jkls} in a more general framework. In the present work, we focus on the case where both $\mathbb{M}$ and $\mathbb{L}$ are powers of some strongly regular sequence admitting a nonzero proximate order, the sequence of quotients being a positive power of the initial sequence, which turns out to be a strongly regular sequence (see Lemma~\ref{lema1}) admitting a nonzero proximate order (see Remark 4.8 (i),~\cite{sanz}). 

The iterated procedure approach to multisummability in the more general context of strongly regular sequences reads as follows.

\begin{defin}[Definition 4.22,~\cite{jkls}] \label{defi257}
Let $\mathbb{M}_j$, where $j=1,2$, be two strongly regular sequences admitting nonzero proximate orders. We assume that $\omega(\mathbb{M}_1)<\omega(\mathbb{M}_2)<2$. For $j=1,2$, we consider a strong kernel $e_j$ of $\mathbb{M}_j$-summability and its associated sequence of moments $m_j$. The formal power series $\hat{f}=\sum_{p\ge0}a_pz^p\in\mathbb{E}[[z]]$ is \emph{$(\mathbb{M}_1,\mathbb{M}_2)$-summable in the multidirection $(d_1,d_2)\in\R^2$} with $|d_1-d_2|<\pi(\omega(\mathbb{M}_2)-\omega(\mathbb{M}_1))/2$ if:
\begin{itemize}
\item[(i)] $\hat{g}=\hat{\mathcal{B}}_{m_1,z}(\hat{f}(z))$ is $\mathbb{M}_2/\mathbb{M}_1$-summable along direction $d_2$. Let $g$ denote such $\mathbb{M}_2/\mathbb{M}_1$-sum.
\item[(ii)] $g$ admits analytic continuation $g_1$ in an infinite sector $S_{d_1}$ of bisecting direction $d_1$ with $g_1\in\mathcal{O}^{\mathbb{M}_1}(S_{d_1},\mathbb{E})$.
\end{itemize}
The $(\mathbb{M}_1,\mathbb{M}_2)$-sum of $\hat{f}$ in the multidirection $(d_1,d_2)$ is given by $T_{e_1,d_1}g_1$, which determines a holomorphic function on a bounded sector of bisecting direction $d$ and opening slightly larger than $\omega(\mathbb{M}_1)\pi$ (see Proposition~\ref{prop2}). We denote the $(\mathbb{M}_1,\mathbb{M}_2)$-sum of $\hat{f}$ in the multidirection $(d_1,d_2)$ by $\mathcal{S}_{(\mathbb{M}_1,\mathbb{M}_2),(d_1,d_2)}(\hat{f})$, and $\mathbb{E}\{z\}_{(\mathbb{M}_1,\mathbb{M}_2),(d_1,d_2)}$ stands for the set of all formal power series $\hat{f}(z)\in\mathbb{E}[[z]]$ which are $(\mathbb{M}_1,\mathbb{M}_2)$-multisummable along the multidirection $(d_1,d_2)$.
\end{defin}

It is worth remarking that $\mathbb{M}$-summability along a direction is stated in~\cite{jkls} in terms of weight sequences satisfying less restrictive conditions than strongly regular sequences. However, any weight sequence admitting a nonzero proximate order is indeed a strongly regular sequence. According to Theorem 4.23~\cite{jkls}, we recall that the previous construction does not depend on the kernels for $\mathbb{M}_j$-summability considered in the process, $j=1,2$. Indeed, an equivalent definition of multisummability is the following.

\begin{prop}[Definition 4.1,~\cite{jkls}] \label{defi256}
In the situation of Definition~\ref{defi257}, the formal power series $\hat{f}$ is $(\mathbb{M}_1,\mathbb{M}_2)$-summable in the multidirection $(d_1,d_2)$ if there exist a formal power series $\hat{f}_1$ which is $\mathbb{M}_1$-summable in $d_1$ and a formal power series $\hat{f}_2$ which is $\mathbb{M}_2$-summable in $d_2$ such that $\hat{f}=\hat{f}_1+\hat{f}_2$. Moreover, the $(\mathbb{M}_1,\mathbb{M}_2)$-sum of $\hat{f}$ in the multidirection $(d_1,d_2)$ is given by $\mathcal{S}_{\mathbb{M}_1,d_1}(\hat{f}_1)+\mathcal{S}_{\mathbb{M}_2,d_2}(\hat{f}_2)$.
\end{prop}

The splitting of $\hat{f}$ into a sum in Proposition~\ref{defi256} is essentially unique (see Proposition 4.2,~\cite{jkls}). This equivalent definition of multisummability allows to give a direct proof of the following result.

\begin{lemma}\label{lema296}
Let $\hat{f}(z)$ be a formal power series, and let $k\in\N_0$. We define the formal power series $\hat{g}(z):=z^k \hat{f}(z)$. 

Let $\mathbb{M}_j$, where $j=1,2$, be two strongly regular sequences admitting nonzero proximate orders. We assume that $\omega(\mathbb{M}_1)<\omega(\mathbb{M}_2)<2$. We choose $(d_1,d_2)\in\R^2$ with $|d_1-d_2|<\pi(\omega(\mathbb{M}_2)-\omega(\mathbb{M}_1))/2$. Then, the following statements are equivalent:
\begin{itemize}
\item The formal power series $\hat{f}$ is $(\mathbb{M}_1,\mathbb{M}_2)$-summable in the multidirection $(d_1,d_2)$.
\item The formal power series $\hat{g}$ is $(\mathbb{M}_1,\mathbb{M}_2)$-summable in the multidirection $(d_1,d_2)$.
\end{itemize}
If one of the previous equivalent statements hold, then one has that 
\begin{equation}\label{e283}
\mathcal{S}_{(\mathbb{M}_1,\mathbb{M}_2),(d_1,d_2)}(\hat{g})=z^k\mathcal{S}_{(\mathbb{M}_1,\mathbb{M}_2),(d_1,d_2)}(\hat{f}).
\end{equation}

\end{lemma}

\begin{proof}
If $\hat{f}$ is $(\mathbb{M}_1,\mathbb{M}_2)$-summable in the multidirection $(d_1,d_2)$, then $\hat{f}=\hat{f}_1+\hat{f}_2$, with $\hat{f}_1$ being $\mathbb{M}_1$-summable in $d_1$ and $\hat{f}_2$ being $\mathbb{M}_2$-summable in $d_2$. Then, $z^k\hat{f}_1$ is $\mathbb{M}_1$-summable in $d_1$ and $z^k\hat{f}_2$ is $\mathbb{M}_2$-summable in $d_2$. Then, $\hat{g}=z^k\hat{f}=z^k\hat{f}_1+z^k\hat{f}_2$, which entails that $\hat{g}$ is $(\mathbb{M}_1,\mathbb{M}_2)$-summable in the multidirection $(d_1,d_2)$.

On the other hand, if $\hat{g}$ is $(\mathbb{M}_1,\mathbb{M}_2)$-summable in the multidirection $(d_1,d_2)$, then $\hat{g}=\hat{g}_1+\hat{g}_2$, with $\hat{g}_1$ being $\mathbb{M}_1$-summable in $d_1$ and $\hat{g}_2$ being $\mathbb{M}_2$-summable in $d_2$. As the splitting is essentially unique, and $\hat{g}(z):=z^k \hat{h}(z)$, one can choose a splitting in which $\hat{g}_j=z^k\hat{h}_j$, for $j=1,2$. Lemma~\ref{lema100} guarantees that the formal power series $\hat{g}_j$ is $\mathbb{M}_j$-summable along direction $d_j$ iff $\hat{h}_j$ is $\mathbb{M}_j$-summable along direction $d_j$. Then, $\hat{f}$ can be written in the form $\hat{f}=\hat{h}_1+\hat{h}_2$, where $\hat{h}_j$ is $\mathbb{M}_j$-summable along direction $d_j$, for $j=1,2$, leading to multisummability of $\hat{f}$.  

Regarding the construction of the sums above, one also arrives at (\ref{e283}).

\end{proof}

\section{On moment differentiation}\label{secmomdif}

In this section, we focus our attention on the concept of a moment derivative and recall some properties associated with this notion. We also state some new results to be applied in the work. 

The notion of a generalized derivative operator allows to consider functional problems under greater generality. More precisely, we deal with the following formal operator.

\begin{defin}\label{defi:derivative}
Let $(\mathbb{E},\left\|\cdot\right\|_{\mathbb{E}})$ be a complex Banach space. For any fixed sequence of moments $(m_e(p))_{p\ge 0}$ we define the \emph{$m$-differential operator} $\partial_{m_e,t}\colon\mathbb{E}[[t]]\to\mathbb{E}[[t]]$ by the formula:
\begin{equation*}
\partial_{m_e,z}\left(\sum_{p\ge 0}\frac{a_{p}}{m_e(p)}z^{p}\right):=\sum_{p\ge 0}\frac{a_{p+1}}{m_e(p)}z^{p}.
\end{equation*}
\end{defin}

Usual derivatives are recovered when considering the moment sequence $(p!)_{p\ge0}$, which is associated with the kernel function $e(z)=z\exp(-z)$. In addition to this, moment derivatives can be read in terms of Caputo $\alpha$-fractional derivatives $\partial_{z}^{\alpha}$ as follows. Let $\alpha$ be a positive rational number. The moment sequence $m_{\alpha}:=(\Gamma(1+\alpha p))_{p\ge0}$ is associated with the kernel function $e(z)=\frac{1}{\alpha}z^{\frac{1}{\alpha}}\exp(-z^{\frac{1}{\alpha}})$ and the fractional derivative of order $\alpha$ is defined on formal power series in $z^{\alpha}$ by
$$\partial_{z}^{\alpha}\left(\sum_{p\ge0}\frac{a_p}{\Gamma(1+\alpha p)}z^{\alpha p}\right)=\sum_{p\ge0}\frac{a_{p+1}}{\Gamma(1+\alpha p)}z^{\alpha p}.$$
Therefore, one has 
$$(\partial_{m_{\alpha}}\hat{f})(z^{\alpha})=\partial_{z}^{\alpha}(\hat{f}(z^{\alpha})),$$
for every $\hat{f}\in\mathbb{E}[[z]]$. %

We may also observe that $q-$derivatives defined by
$$D_{q,z}f(z)=\frac{f(qz)-f(z)}{qz-z}$$
for some fixed $q\in (0,1)$ can be interpreted in terms of the moment derivatives associated with the moment sequence $([p]_{q}!)_{p\ge0}$, with $[p]_{q}!=[1]_q[2]_q\cdots [p]_{q}$ and $[j]_{q}=\sum_{h=0}^{j-1}q^h$. 

We also have the following result.

\begin{lemma}[Lemma 3,~\cite{LMS2}]\label{lema347}
Let $m_j=(m_j(p))_{p\ge0}$ for $j=1,2$ be two sequences of moments. Then,
\begin{itemize}
\item[-] The sequence $m_1m_2=(m_1(p)m_2(p))_{p\ge0}$ is a sequence of moments.
\item[-] $\mathcal{B}_{m_1,z}\circ\partial_{m_2,z}\equiv \partial_{m_1m_2,z}\circ\hat{\mathcal{B}}_{m_1,z}$ as operators defined in $\mathbb{E}[[z]]$.
\end{itemize}
\end{lemma}

Moment differentiation can be naturally extended to holomorphic functions on some neighborhood of the origin by identifying the function with its Taylor series at the origin. However, this formal differentiation does not preserve convergence unless some regularity property is assumed for the sequence of moments. However, if one departs from a strongly regular sequence $\mathbb{M}$ which admits a nonzero proximate order, and considers a pair of kernel functions associated with it and then constructs the corresponding sequence of moments $m$, then it holds that $\mathbb{M}$ and $m$ generate the same ultraholomorphic space of functions (i.e., they are equivalent sequences) and $m$ is indeed a strongly regular sequence (see Remark 3.8,~\cite{lastramaleksanz}). In~\cite{LMS2}, the definition of moment differentiation was also extended to the generalized sum along a direction of a formal power series as the generalized sum along that same direction of the formal moment derivative of the initial formal power series. For that purpose, the moment derivative was also provided for any function defined on some neighborhood of the origin with holomorphic extension to an infinite sector and with certain generalized exponential growth at infinity (as described in Definition~\ref{defi3}). Indeed, the first part of Theorem 3~\cite{LMS2} reads as follows.

\begin{theo}
Let $m_e=(m_e(p))_{p\ge0}$ be a sequence of moments. We also fix $d,\theta,r\in\R$ with $\theta,r>0$ and $\varphi\in\mathcal{O}(\hat{S}_d(\theta;r),\mathbb{E})$. Then, there exists $0<\tilde{r}<r$ such that for all $0<\theta_1<\theta$, all $z\in\hat{S}_{d}(\theta_1;\tilde{r})$ and all $n\in\N_0$, one has that
\begin{equation}\label{e268}
\partial_{m_e,z}^{n}\varphi(z)=\frac{1}{2\pi i}\oint_{\Gamma_{z}}\varphi(\omega)\int_0^{\infty(\tau)}\xi^n E(z\xi)\frac{e(\omega \xi)}{\omega \xi}d\xi d\omega,
\end{equation}
with $\tau=\tau(\omega)\in(-\arg(\omega)-\frac{\omega(m_{e})\pi}{2},-\arg(\omega)+\frac{\omega(m_{e})\pi}{2})$. The integration path $\Gamma_z$ is a deformation of the circle $\{|\omega|=r_1\}$, for any choice of $0<r_1<r$, which depends on $z$. More precisely, such deformation consists of substituting some arc of the circle contained in $\hat{S}_{d}(\theta;r)$ by a simple path which attains an adequate sufficient distance to the origin while it remains inside $S_{d}(\theta)$.
\end{theo}

At this point, we give a step forward in order to define the moment derivatives on functions defined on some sectorial region, and which can be extended under certain generalized exponential growth to infinity.

\begin{theo}\label{teo2}
 Let $m_e=(m_e(p))_{p\ge 0}$ be a sequence of moments, and $\M=(M_p)_{p\ge 0}$ be a strongly regular sequence admitting a nonzero proximate order. Let $d_1,d_2\in\R$ satisfying
 $|d_1-d_2|<a \frac{\omega(\M)\pi}{2}$ for some $a>0$. We choose $\hat{u}\in\EE\{z\}_{\M^a,d_2}$ and write $u=S_{\M^a,d_2}(\hat{u})\in\Oo(G,\EE)$, for some sectorial region $G=G_{d_2}(\theta)$ with $\theta>a\pi\omega(\M)$, as seen in Figure \ref{fig2}. Assume moreover that $u$ can be extended (the extension is also denoted by $u$) to an infinite sector of bisecting direction $d_1$, with $u\in\Oo^{\M^b}(S_{d_1},\EE)$. Then, the following statements hold:
\begin{itemize}
\item[(a)] There exists $\tilde{r}>0$ such that for every $S'\prec S_{d_1}$ and $G'\prec G\cap D(0,\tilde{r})$ and all $z\in \tilde{S}:=S'\cup G'$ and $n\in\N_0$
$$\partial_{m_e,z}^{n}u(z)=\frac{1}{2\pi i}\oint_{\Lambda_z}u(w)\int_{0}^{\infty(\tau)}\xi^n E(z\xi)\frac{e(w \xi)}{w \xi}d\xi dw,$$
where $\tau=\tau(\omega)\in (-\arg(w)-\frac{\omega(m_{e})\pi}{2},-\arg(w)+\frac{\omega(m_{e})\pi}{2})$. The path $\Lambda_z$ depends on $z$.
\item[(b)] There exist $C_4,C_5,C_6>0$ such that 
\begin{equation}\label{e286}
 \|\partial_{m_e,z}^n u(z)\|_\EE\le C_4 C_5^n m_e(n)M^a_n \exp\left(M^b(C_6|z|)\right)\quad \textrm{ for all }n\in\N_0 \textrm{ and }z\in \tilde{S},
\end{equation}
 where $M^b(t)$ is the function defined in \eqref{eq:function_M}, corresponding to the strongly regular sequence $\mathbb{M}^b$.
\end{itemize}
\end{theo}

\begin{figure}
\includegraphics{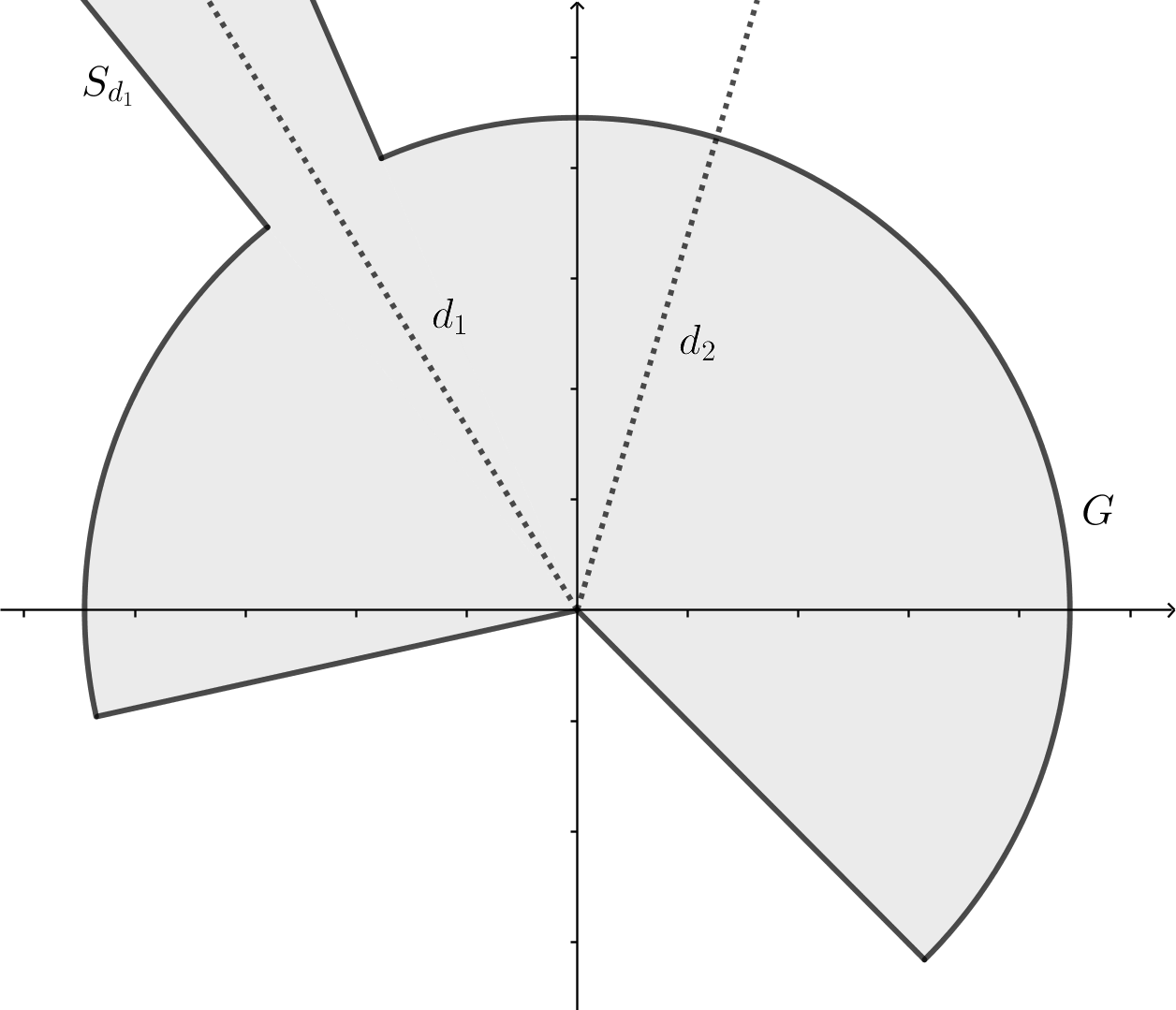}\caption{Example of configuration of the sets $S_{d_1}$ and $G$}\label{fig2}
\end{figure}

\begin{proof}
The proof of the previous result is based on that of Theorem 3 and Proposition 2,~\cite{LMS2}. We provide a complete proof for the sake of completeness and focus on distinctive points with respect to the results in that previous work.

Let $S'\prec S_{d_1}$ and $G''\prec G$. Let $r_1>0$ and $\theta>\theta'>0$ such that $S_{d_2}(\theta;2r_1)\subseteq G$ and $S'\cup G''\subset S_{d_2}(\theta')$. We take $\tilde{r}:=\frac{\tilde{k}r_1}{K\rho(2)}$, where $\tilde{k}$ is given in (\ref{kernel2}), $K$ is defined in (\ref{e162}), and $\rho(\cdot)$ is as in Lemma~\ref{lema0}.
Let $G':=G''\cap D(0,\tilde{r})$ and $\tilde{S}:=S'\cup G'$. Choose $z\in \tilde{S}$. The path $\Lambda_z$ is constructed as follows.
 Let $r_1e^{i\theta_1},r_1e^{i\theta_2}$ be the points in $\{z\in\C:|z|=r_1\}\cap S_d$ and  let $\tilde{P}=r_1e^{i\tilde{\theta}_1},\tilde{Q}=r_1e^{i\tilde{\theta}_2}$ be the points in $\{z\in\C:|z|=r_1\}\cap S'$. We assume $\theta_1<\tilde{\theta}_1<\tilde{\theta}_2<\theta_2$. We define $\Lambda_1:=[0,r_1]e^{i(d_2-\theta/2)}$, $\Lambda_2$ is the arc of the circle of radius $r_1$ from $r_1e^{i(d_2-\theta/2)}$ to $r_1e^{i(\theta_1+\tilde{\theta}_1)/2}$. We also put $\Lambda_3:=[r_1e^{i(\theta_1+\tilde{\theta}_1)/2},Re^{i(\theta_1+\tilde{\theta}_1)/2}]$, with $R=R(z)>0$ to be determined. $\Lambda_4$ is the arc of circle from $Re^{i(\theta_1+\tilde{\theta}_1)/2}$ to $Re^{i(\theta_2+\tilde{\theta}_2)/2}$, $\Lambda_5:=[r_1e^{i(\theta_2+\tilde{\theta}_2)/2},Re^{i(\theta_2+\tilde{\theta}_2)/2}]$, $\Lambda_6$ is the arc of the circle of radius $r_1$ from $r_1e^{i(\theta_2+\tilde{\theta}_2)/2}$ to $r_1e^{i(d_2+\theta/2)}$ and $\Lambda_7:=[0,r_1]e^{i(d_2+\theta/2)}$. We finally define the integration path
$$\Lambda_z:=\Lambda_1+\Lambda_2+\Lambda_3+\Lambda_4-\Lambda_5+\Lambda_6-\Lambda_7,$$
see Figure~\ref{fig3}. In case that $|z|< \tilde{r}=\frac{\tilde{k}r_1}{K\rho(2)}$, where $\tilde{k}$ is given in (\ref{kernel2}), $K$ is defined in (\ref{e162}), and $\rho(\cdot)$ is as in Lemma~\ref{lema0}, then one can choose $\tilde{P}=\tilde{Q}$ and remove $\Lambda_3$, $\Lambda_4$ and $\Lambda_5$ from the concatenation. Otherwise, $R:=\frac{\rho(2)K}{\tilde{k}}|z|$.

\begin{figure}
\includegraphics{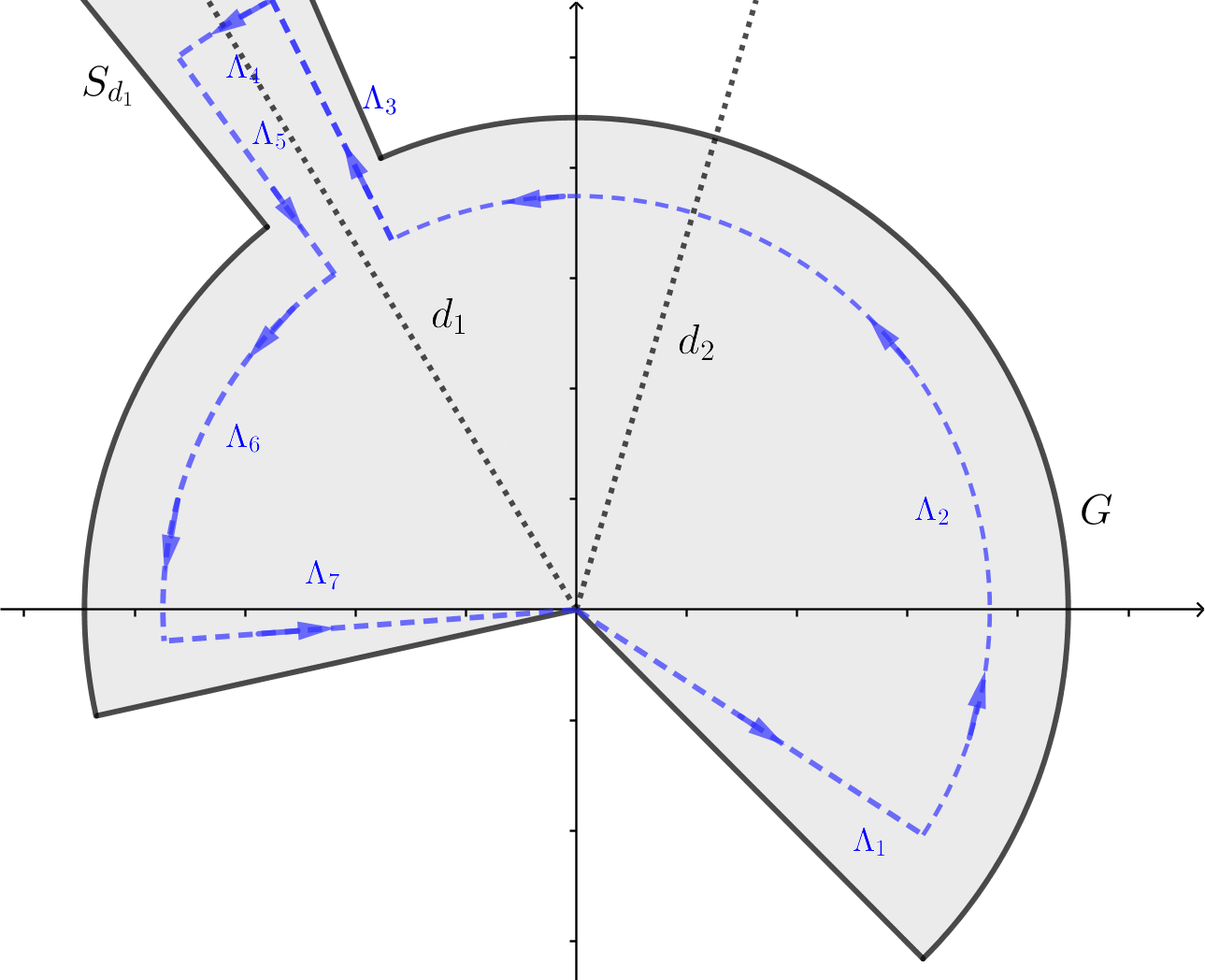}\caption{Integration path $\Lambda_z$}\label{fig3}
\end{figure}

For the first part of the proof, we observe from (35) in~\cite{sanzproceedings} that
$$\int_{0}^{\infty(\tau)}E(z\xi)\frac{e(w \xi)}{w\xi}d\xi=\frac{1}{w-z},$$ 
for every pair of complex numbers $(z,w)$ in which both sides of the previous expression are defined. Hence, if $u\in\mathcal{O}(G\cup S_{d_1})$ then one may replace the contour $\Gamma_z$ in (\ref{e268}) by $\Lambda_z$.

We now provide the estimates in (\ref{e286}). We first observe (see (17) in~\cite{LMS2}) that 
\begin{equation}\label{e423}
\left|\int_{0}^{\infty(\tau)}\xi^nE(z\xi)\frac{e(\omega \xi)}{\omega \xi}d\xi\right|\le A_0 B_0^n m_e(n),\qquad n\ge0,
\end{equation}
valid for all $z\in \C$ with $|z|\le \tilde{r}$ and $\tau\in\left(-\arg(\omega)-\frac{\omega(m_e)\pi}{2},-\arg(\omega)+\frac{\omega(m_e)\pi}{2}\right)$, for $\omega\in\C$ with $|\omega|=r_1$. This entails there exist $A_{11},B_{11}>0$ such that
$$
\left\|\frac{1}{2\pi i}\int_{\Lambda_j}u(\omega)\int_0^{\infty(\tau)}\xi^n E(z\xi)\frac{e(\omega \xi)}{\omega \xi}d\xi d\omega\right\|_{\mathbb{E}}\le \left(\sup_{|\omega|=r_1,\omega\in S_{d_2}(\theta;2r_1)}\left\|u(\omega)\right\|\right)A_{11}B_{11}^nm_e(n),
$$
for $j=2,6$, valid for every $n\in\N_0$. We recall that in this case $\Lambda_3,\Lambda_4$ and $\Lambda_5$ do not appear in the integration path. It only remains to give upper bounds regarding the paths $\Lambda_1$ and $\Lambda_7$, which are postponed.

If $|z|>\tilde{r}$, for every $\omega\in\Lambda_2\cup\Lambda_6$ one can choose $\tau$ with 
$$
\tau\in\left(-\arg(\omega)-\frac{\omega(m_e)\pi}{2},-\arg(\omega)+\frac{\omega(m_e)\pi}{2}\right)\cap \left(-\arg(z)+\frac{\omega(m_e)\pi}{2},-\arg(z)+2\pi-\frac{\omega(m_e)\pi}{2}\right).$$
This entails (see (22)--(26) in the proof of Theorem 3~\cite{LMS2}) the existence of $A_{12},B_{12}>0$ with
$$
\left\|\frac{1}{2\pi i}\int_{\Lambda_j}u(\omega)\int_0^{\infty(\tau)}\xi^n E(z\xi)\frac{e(\omega \xi)}{\omega \xi}d\xi d\omega\right\|_{\mathbb{E}}\le \left(\sup_{|\omega|=r_1,\omega\in S_{d_2}(\theta;2r_1)}\left\|u(\omega)\right\|\right)A_{12}B_{12}^nm_e(n),
$$
for $j=2,6$, and all $n\in\N_0$. The previous expression can be bounded from above by
$$A_{12}B_{12}^nm_e(n)\exp(M^b(C_{12}|z|))$$ for some $C_{12}>0$ taking into account the growth at infinity determined by $v$ and the choice of the integration paths. On the other hand, parametrizing the integration path $\Lambda_{3}$ and analogous estimates as in the previous expression yield
$$
\left\|\frac{1}{2\pi i}\int_{\Lambda_3}u(\omega)\int_0^{\infty(\tau)}\xi^n E(z\xi)\frac{e(\omega \xi)}{\omega \xi}d\xi d\omega\right\|_{\mathbb{E}}\le A_{13}B_{13}^nm_e(n)\exp(M^b(C_{13}|z|)),
$$
for some $A_{13},B_{13},C_{13}>0$ (see (28) in~\cite{LMS2}). These same upper bounds hold for the integration along $\Lambda_5$. Again, the parametrization of $\Lambda_{4}$ and usual estimates yield 
$$
\left\|\frac{1}{2\pi i}\int_{\Lambda_4}u(\omega)\int_0^{\infty(\tau)}\xi^n E(z\xi)\frac{e(\omega \xi)}{\omega \xi}d\xi d\omega\right\|_{\mathbb{E}}\le A_{14}B_{14}^nm_e(n)\exp(M^b(C_{14}|z|)),
$$
for some $A_{14},B_{14},C_{14}>0$ (see (31) in~\cite{LMS2}).

At this point, it only remains to provide upper bounds for 
\begin{equation}\label{e345}
\left\|\frac{1}{2\pi i}\int_{\Lambda_j}u(\omega)\int_0^{\infty(\tau)}\xi^n E(z\xi)\frac{e(\omega \xi)}{\omega \xi}d\xi d\omega\right\|_{\mathbb{E}}
\end{equation}
for $j=1,7$ for all $z\in \tilde{S}$. Let $n\in\N_0$. We write $\hat{u}(z)=\sum_{p\ge0}u_pz^{p}$ and define 
$$v(z):=\frac{1}{z^n}\left(u(z)-\sum_{p=0}^{n-1}u_pz^p\right).$$

\begin{lemma}\label{lema3}
The function $v$ belongs to $\mathcal{O}^{\mathbb{M}^b}(S_{d_1},\mathbb{E})$. In addition to this, for every $S'\prec S_{d_1}$ there exist $\tilde{C},\tilde{B},\tilde{A}>0$, which do not depend on $n$, such that
$$\left\|v(z)\right\|_{\mathbb{E}}\le \tilde{C}\tilde{A}^nM_n^a\exp\left(M^b\left(\frac{|z|}{\tilde{B}}\right)\right),$$
for all $z\in S'$.
\end{lemma} 
\begin{proof}
It is clear that $v\in\mathcal{O}^{\mathbb{M}^b}(S_{d_1},\mathbb{E})$. We have that for all $S'\prec S_{d_1}$ one has that (\ref{e152}) holds for some $C,K>0$.

Bearing in mind that $u$ admits $\hat{u}$ as its asymptotic expansion, there exist $K_1,\tilde{C}_1,\tilde{A}_1>0$ such that 
\begin{equation}\label{e364}
\left\|v(z)\right\|_{\mathbb{E}}\le \tilde{C}_1\tilde{A}_1^nM^a_n,
\end{equation}
for all $z\in S'$ with $|z|\le K_1$, provided that the opening of $S_{d_1}$ is small enough.
On the other hand, for every $z\in S'$ with $|z|\ge K_1$ and all $0\le p\le n-1$, one can apply Lemma~\ref{lema140} and usual estimates to arrive at
$$\left\|u_pz^{p-n}\right\|_{\mathbb{E}}\le\frac{1}{K_1^{n-p}}\tilde{C}_1\tilde{A}_1^pM_p^a\le \tilde{C}_2\tilde{A}_2^nM_n^a.$$
Also,
\begin{equation}\label{e381}
\frac{\left\|u(z)\right\|_{\mathbb{E}}}{|z|^{n}}\le \frac{1}{K_1^n}C\exp(M^b(|z|/K))
\end{equation}
for all $z\in S'$ with $|z|\ge K_1$.
The result follows from here.
\end{proof}

Observe that for every $n\in\N_0$ one has that
\begin{equation}\label{e389}
\partial_{m_e,z}^{n}u(z)=\partial_{m_e,z}^{n}\left(u(z)-\sum_{p=0}^{n-1}u_pz^p\right)=\partial_{m_e,z}^{n}(z^nv(z)).
\end{equation}

In view of (\ref{e389}) one may substitute the study of upper estimates of (\ref{e345}) by those of
\begin{equation}\label{e345b}
I_j:=\left\|\frac{1}{2\pi i}\int_{\Lambda_j}v(\omega)\int_0^{\infty(\tau)}\xi^n\omega^n E(z\xi)\frac{e(\omega \xi)}{\omega \xi}d\xi d\omega\right\|_{\mathbb{E}}.
\end{equation}
We provide upper bounds for (\ref{e345b}) for $j=1$, which remain valid for the case $j=7$.


In view of Lemma~\ref{lema3}, one derives
\begin{equation}\label{e411}
\left\|v(\omega)\right\|_{\mathbb{E}}\le \tilde{C}_1\tilde{A}_1^nM_n^{a}\exp\left(M^b\left(\frac{|\omega|}{\tilde{B}}\right)\right)\le \tilde{C}_1\exp\left(M^b\left(\frac{r_1}{\tilde{B}}\right)\right)\tilde{A}_1^nM_n^{a}.
\end{equation}

By Proposition 2, \cite{LMS2} there exists $\varepsilon > 0$, such that the estimation (\ref{e286}) holds for $z\in \tilde{S}\cap D(0,\varepsilon)$. Hence, we may assume that $|z|\geq\varepsilon$.
Let $M_E$ be a positive constant given in Definition \ref{defin:kernel-functions}.

We split the path of integration in the inner integral in (\ref{e345b}) into the segment
$[0, e^{i\tau}M_E/|z| ]$ and the ray $[e^{i\tau}M_E/|z| ,\infty(\tau)]$,
and we interchange the order of integration in (\ref{e345b}). Next, we observe that $\hbox{arg}(\omega)=d_2-\theta/2$ and $\hbox{arg}(\xi)=\tau\in(-\hbox{arg}(\omega)-\frac{\pi \omega(m_e)}{2},-\hbox{arg}(\omega)+\frac{\pi \omega(m_e)}{2})$. Hence
\begin{multline}\label{eq:I_1}
 I_1\leq \frac{1}{2\pi}\sup_{\omega\in \Lambda_1}\|v(\omega)\|_{\mathbb{E}}
 \bigg(\int_0^{M_E/|z|}|E(z se^{i\tau})|\int_{0}^{r_1}s^nw^n |e(wse^{i\tilde{\theta}})|\frac{dw}{w} \frac{ds}{s}\\ 
  + \int_{M_E/|z|}^{\infty}|E(z se^{i\tau})|\int_{0}^{r_1}s^nw^n |e(wse^{i\tilde{\theta}})|\frac{dw}{w} \frac{ds}{s}\bigg)=:\frac{1}{2\pi}\sup_{\omega\in \Lambda_1}\|v(\omega)\|_{\mathbb{E}}(I_{11}+I_{12}),
\end{multline}
for $\tilde{\theta}:=d_2-\theta/2+\tau\in(-\pi\omega(m_e)/2,\pi\omega(m_e)/2)$.

Observe that
\begin{equation}\label{e414}
\int_{0}^{r_1}s^nw^n|e(wse^{i\tilde{\theta}})|\frac{dw}{w}\le  
\int_{0}^{sr_1}t^{n-1}|e(te^{i\tilde{\theta}})|dt.
\end{equation}

If $s<M_E/|z|$ then using (\ref{eq:est_e})  we continue the estimation (\ref{e414}) by
\begin{equation}\label{eq:est_1_case}
\int_{0}^{r_1s}t^{n-1}|e(te^{i\tilde{\theta}})|dt\le C \int_{0}^{r_1s}t^{n-1}t^{\alpha} dt\le\frac{C r_1^{n+\alpha}}{n+\alpha}s^{n+\alpha}\leq\tilde{C}_2\tilde{A}_2^ns^{n+\alpha},
\end{equation}
for some $\tilde{C}_2,\tilde{A}_2>0$ and $\alpha>0$.

By (\ref{eq:est_1_case}) we get
\begin{equation}\label{eq:I_11}
 I_{11}\leq \sup_{|\zeta|\le M_E}|E(\zeta)|\int_0^{M_E/\varepsilon}\tilde{C}_2\tilde{A}_2^ns^{n+\alpha-1}ds\le \tilde{C}_3\tilde{A}_3^n,
\end{equation}
for some $\tilde{C}_3,\tilde{A}_3>0$.

On the opposite case $s\geq M_E/|z|$,
using Lemma~\ref{lema_1} and Lemma~\ref{lema_2} and (\ref{e162}) we estimate (\ref{e414}) by
\begin{equation}\label{eq:est_2_case}
\int_{0}^{r_1s}t^{n-1}|e(te^{i\tilde{\theta}})|dt \le \tilde{C}_4\tilde{A}_4^nm_e(n),
\end{equation}
for some constants $\tilde{C}_4,\tilde{A}_4>0$.

Since $\theta'<\theta$ and $\hbox{arg}(z)\in (d_2-\theta'/2, d_2+\theta'/2)$, there exists $\delta>0$
such that for every $z\in \tilde{S}$ there exists $\tau\in(-d_2+\theta/2-\pi\omega(m_e)/2,-d_2+\theta/2+\pi\omega(m_e)/2)$ satisfying
$\hbox{arg}(z)+\tau \not\in (-\pi\omega(m_e)/2-\delta,\pi\omega(m_e)/2+\delta)$. Hence, using (\ref{e202}) and (\ref{eq:est_2_case}) we estimate
\begin{equation}\label{eq:I_12}
 I_{12}\leq \tilde{C}_4\tilde{A}_4^nm_e(n)\frac{\tilde{c}_2}{|z|^{\beta}}\int_{M_e/|z|}^{\infty}\frac{1}{s^{\beta+1}}ds\le \tilde{C}_5\tilde{A}_5^n m_e(n),
 \end{equation}
for some $\tilde{C}_5,\tilde{A}_5 >0$ and $\beta>0$.

Taking into account (\ref{e411}), (\ref{eq:I_1}), (\ref{eq:I_11}) and (\ref{eq:I_12}), upper estimates as above yield 
$$I_1\le \tilde{C}_6\tilde{C}_7^nm_e(n)M_n^a\exp\left(M^b(\tilde{C}_8|z|)\right),$$
for some $\tilde{C}_j>0$, $6\le j\le 8$.

This concludes the proof of (\ref{e286}).
\end{proof}

In the last part of this section, we describe compatibility conditions regarding asymptotic expansions and moment derivation, which allows to provide a differential structure to the set of multisummable formal power series.

\begin{lemma}\label{lema484}
In the situation of Theorem~\ref{teo2}, one has that $\partial_{m_e,z}u(z)$ is the $\mathbb{M}^a$-sum of $\partial_{m_e,z}\hat{u}(z)$ in $G$.
\end{lemma} 
\begin{proof}
Observe that $E_{m_e}(z\xi)=\sum_{n\ge0}\frac{(z\xi)^n}{m_e(n)}$. Therefore, for every $p\ge0$ one has
$$\lim_{z\to 0}\partial_{m_e,z}^{p}E_{m_e}(z\xi)=\frac{\xi^p}{m_{e}(0)}.$$
Analogously, 
$$\lim_{z\to 0}\partial_{z}^{p}E_{m_e}(z\xi)=\frac{p!}{m_e(p)}\xi^p,$$
which yields
$$\lim_{z\to 0}\partial_{m_e,z}^{p}E_{m_e}(z\xi)=\frac{m_e(p)}{m_e(0)p!}\lim_{z\to 0}\partial_{z}^{p}E_{m_e}(z\xi).$$
In view of (\ref{e268}), this means that given $u$ which admits $\hat{u}(z)=\sum_{n\ge0}u_nz^n$ as its asymptotic expansion at the origin in $G$, then
\begin{align*}
\lim_{z\to 0,z\in G}\partial_{m_e,z}^{p}u(z)&=\frac{m_e(p)}{m_e(0)p!}\lim_{z\to 0,z\in G}\partial_{z}^{p}u(z)\\
&=\frac{m_e(p)}{m_e(0)p!}\partial_{z}^{p}\hat{u}(z)\left.\right|_{z=0}=\frac{m_e(p)}{m_e(0)p!}p!u_p=\frac{m_e(p)}{m_e(0)}u_p.
\end{align*}
On the other hand,
$$\partial_{m_e,z}^{p}\hat{u}(z)=\sum_{n\ge0}\frac{u_{n+p}m_e(n+p)}{m_e(n)}z^n,$$
which entails that
$$\partial_{m_e,z}^{p}\hat{u}(z)\left.\right|_{z=0}=u_p\frac{m_e(p)}{m_e(0)}.$$
We conclude that
$$\lim_{z\to 0,z\in G}\partial_{m_e,z}^{p}u(z)=\partial_{m_e,z}^{p}\hat{u}(z)\left.\right|_{z=0}.$$
This is an equivalent condition for $\partial_{m_e,z}u(z)$ to admit $\partial_{m_e,z}\hat{u}(z)$ as its asymptotic expansion in $G$.
\end{proof}

\begin{corol}\label{coro519}
Let $m_e=(m_e(p))_{p\ge0}$ be a sequence of moments, and let $\mathbb{M}$ be a strongly regular sequence which admits a nonzero proximate order. We also choose positive numbers $a,b$ such that $\omega(\mathbb{M})<2/(a+b)$. Given any multidirection $(d_1,d_2)\in\R^2$ such that $|d_1-d_2|<a\frac{\omega(\mathbb{M})\pi}{2}$, the space $\mathbb{E}\{z\}_{(\mathbb{M}^{b},\mathbb{M}^{a+b}),(d_1,d_2)}$ is closed under $m_e$-differentiation.
\end{corol}
\begin{proof}
In view of Lemma~\ref{lema1}, we have that $\mathbb{M}^{b}$ is a strongly regular sequence which admits a nonzero proximate order (see Remark 4.8 (i),~\cite{sanz}). This guarantees the existence of a moment sequence $m_{e^b}:=(m_e(p))_{p\ge0}$ associated with some kernel function $e^b$ for $\mathbb{M}^b$-summability. 

Let $\hat{u}\in \mathbb{E}\{z\}_{(\mathbb{M}^{b},\mathbb{M}^{a+b}),(d_1,d_2)}$. It holds that $\hat{\mathcal{B}}_{m_{e^b}}\hat{u}$ is $\mathbb{M}^{a}$-summable in direction $d_2$. Equivalently, there exists $U\in\mathcal{O}(G,\mathbb{E})$, for some sectorial region $G$ of bisecting direction $d_2$ and opening larger than $a\pi\omega(\mathbb{M})$ such that $U$ admits $\hat{\mathcal{B}}_{m_{e^b}}\hat{u}$ as its $\mathbb{M}^{a}$-asymptotic expansion in $G$. Lemma~\ref{lema484} guarantees that $\partial_{m_e\cdot m_{e^b},z}U(z)$ is the $\mathbb{M}^{a}$-sum of $\partial_{m_e\cdot m_{e^b},z}(\hat{\mathcal{B}}_{m_{e^b}}\hat{u})$ in $G$. 

We recall from Lemma~\ref{lema347} that 
$$\partial_{m_e\cdot m_{e^b},z}(\hat{\mathcal{B}}_{m_{e^b}}\hat{u})=\hat{\mathcal{B}}_{m_{e^b}}\left(\partial_{m_e,z}\hat{u}(z)\right).$$
This entails that the formal power series $\hat{\mathcal{B}}_{m_{e^b}}\left(\partial_{m_e,z}\hat{u}(z)\right)$ is $\mathbb{M}^{a}$-summable along direction $d_2$. Moreover, its sum can be extended to an infinite sector $S_{d_1}$ of bisecting direction $d_1$, which belongs to $\mathcal{O}^{\mathbb{M}^{b}}(S_{d_1},\mathbb{E})$ due to the fact that assumptions in Theorem~\ref{teo2} hold, and regarding (\ref{e286}). This allows to conclude that $\partial_{m_e,z}\hat{u}(z)\in\mathbb{E}\{z\}_{(\mathbb{M}^{b},\mathbb{M}^{a+b}),(d_1,d_2)}$.
\end{proof}

As a consequence, a definition of moment derivative for multisums of formal power series can be stated.

\begin{defin}\label{defi624}
Let $m_e=(m_e(p))_{p\ge0}$ be a sequence of moments, and let $\mathbb{M}$ be a strongly regular sequence which admits a nonzero proximate order. We also choose positive numbers $a,b$ such that $\omega(\mathbb{M})<2/(a+b)$. Given any multidirection $(d_1,d_2)\in\R^2$ such that $|d_1-d_2|<a\frac{\omega(\mathbb{M})\pi}{2}$. For every $\hat{f}\in\mathbb{E}\{z\}_{(\mathbb{M}^{b},\mathbb{M}^{a+b}),(d_1,d_2)}$ we define
$$\partial_{m_e,z}(\mathcal{S}_{(\mathbb{M}^{b},\mathbb{M}^{a+b}),(d_1,d_2)}(\hat{f})):=\mathcal{S}_{(\mathbb{M}^{b},\mathbb{M}^{a+b}),(d_1,d_2)}(\partial_{m_e,z}\hat{f}).$$
\end{defin}

We conclude with some properties of multisummable series which will appear in the next section.

\begin{lemma}\label{lema297}
Let $\mathbb{E}$ be the Banach space of holomorphic functions on some nonempty closed neighborhood of the origin, $\overline{D}$, endowed with the supremum norm. Let $\mathbb{M}_j$ for $j=1,2$ be two strongly regular sequences admitting nonzero proximate orders, with $\omega(\mathbb{M}_1)<\omega(\mathbb{M}_2)<2$. Let $(d_1,d_2)\in\R^2$ such that $|d_1-d_2|<\pi(\omega(\mathbb{M}_2)-\omega(\mathbb{M}_1))$. We also consider a moment sequence $m$ and $a(z)\in\mathcal{O}(\overline{D})$.

For every $\hat{f}(t,z)\in\mathbb{E}[[t]]$ with $\hat{f}\in\mathbb{E}\{t\}_{(\mathbb{M}_1,\mathbb{M}_2),(d_1,d_2)}$, one has:
\begin{itemize}
\item $a(z)\hat{f}(t,z)\in\mathbb{E}\{t\}_{(\mathbb{M}_1,\mathbb{M}_2),(d_1,d_2)}$.
\item $\partial_{m,z}\hat{f}\in\mathbb{E'}\{t\}_{(\mathbb{M}_1,\mathbb{M}_2),(d_1,d_2)}$, where $\mathbb{E}'$ is the Banach space of holomorphic functions on some $\overline{D}'\subseteq D$ endowed with the supremum norm.
\end{itemize}
\end{lemma}
\begin{proof}

We write $\mathbb{M}_{j}:=(M_{j,p})_{p\ge0}$ for $j=1,2$. Regarding Proposition~\ref{defi256}, one can write $\hat{f}=\hat{f}_1+\hat{f}_2$, with $\hat{f}_j$ being $\mathbb{M}_j$-summable along direction $d_j$, $j=1,2$. We write $\hat{f}_j(t,z)=\sum_{p\ge0}a_{j,p}(z)t^p\in\mathbb{E}[[t]]$ for $j=1,2$. We observe that 
$$a(z)\hat{f}(t,z)=a(z)\hat{f}_1(t,z)+a(z)\hat{f}_2(t,z).$$
For $j=1,2$, there exists a sectorial region $G_j$ of bisecting direction $d_j$ and opening larger than $\omega(\mathbb{M}_j)\pi$, a holomorphic function $f_j\in\mathcal{O}(G\times \overline{D})$, and positive constants $C,A$ such that for all $t\in G'\prec G$ and all $n\in\N$ one has
\begin{equation}\label{e350}
\left\|f_j(t,z)-\sum_{p=0}^{n-1}a_{j,p}(z)t^{p}\right\|_{\mathbb{E}}\le C A^{n}M_{j,n}|t|^{n},
\end{equation}
for all $(t,z)\in G'\times \overline{D}$. Therefore, the series $a(z)\hat{f}_j(t,z)\in\mathbb{E}[[t]]$ is such that
$$\left\|a(z)f_j(t,z)-\sum_{p=0}^{n-1}a(z)a_{j,p}(z)t^{p}\right\|_{\mathbb{E}}\le C\left(\sup_{z\in\overline{D}}|a(z)|\right) A^{n}M_{j,n}|t|^{n},$$
for all $(t,z)\in G'\times \overline{D}$. This entails that $af_j$ is the $\mathbb{M}_{j}$-sum of $a\hat{f}_{j}$ along direction $d_j$. This concludes that $a(z)\hat{f}(t,z)\in\mathbb{E}\{t\}_{(\mathbb{M}_1,\mathbb{M}_2),(d_1,d_2)}$ again in view of Proposition~\ref{defi256}.

For the second part of the proof, it only remains to check that for all $j=1,2$ the formal power series $\partial_{m,z}\hat{f}_{j}(t,z)\in\mathbb{E}'[[t]]$ is $\mathbb{M}_j$-summable along direction $d_j$, for $j=1,2$. We first observe that the coefficients of $\partial_{m,z}\hat{f}_{j}(t,z)$ as a formal power series in $t$ are holomorphic and bounded functions on some common neighborhood of the origin $D'$, due to the properties of $m$. Let 
$$g_n(t,z)=t^{-n}\left(f_j(t,z)-\sum_{p=0}^{n-1}a_{j,p}(z)t^{p}\right)$$
and define $\mathbb{E}''$ as the Banach space of holomorphic and bounded functions in $G'$ with the sup. norm. 
We observe from (\ref{e350}), which is valid for all $z\in \overline{D'}$, that
$$\left\|t^{-n}\left(f_j(t,z)-\sum_{p=0}^{n-1}a_{j,p}(z)t^{p}\right)\right\|_{\mathbb{E}'}\le C A^{n}M_{j,n},$$
for all $n\in\N$. We have obtained that 
\begin{equation}\label{e577}
\left\|g_n(t,z)\right\|_{\mathbb{E}''}\le CA^nM_{j,n}.
\end{equation} 
We write
$$\left\|t^{-n}\left(\partial_{m,z}f_j(t,z)-\sum_{p=0}^{n-1}(\partial_{m,z}a_{j,p}(z))t^{p}\right)\right\|_{\mathbb{E}'}=\left\|\partial_{m,z}\left(t^{-n}\left(f_j(t,z)-\sum_{p=0}^{n-1}a_{j,p}(z)t^{p}\right)\right)\right\|_{\mathbb{E}'}.$$

An analogous argument as that of Theorem~\ref{teo2} yields 
$$\partial_{m,z}g_n(t,z)=\frac{1}{2\pi i}\int_{|\omega|=\tilde{r}_j} g_n(t,\omega)\int_0^{\infty(\tau)}E(z\xi)\frac{e(\omega\xi)}{\omega}d\xi d\omega,$$
for some $\tilde{r}_j<r$, which can be bounded from above taking into account (\ref{e423}) and (\ref{e577}). This way one arrives at
$$\left\|\partial_{m,z}g_n(t,z)\right\|_{\mathbb{E}''}\le CA_0B_0\tilde{r}_jm(1)A^nM_{j,n},$$
or equivalently
$$\left|\partial_{m,z}f_j(t,z)-\sum_{p=0}^{n-1}(\partial_{m,z}a_{j,p}(z))t^{p}\right|\le CA_0B_0\tilde{r}_jm(1)A^nM_{j,n}|t|^{n},$$
valid for all $(t,z)\in G'\times \overline{D}'$. This yields that $\partial_{m,z}f_j(t,z)$ is the $\mathbb{M}_{j}$-sum of $\sum_{p\ge 0}(\partial_{m,z}a_{j,p}(z))t^{p}$ along direction $d_j$.

\end{proof}

\section{Application. Multisummability of formal solutions to singularly perturbed moment differential equations}\label{secfinal}

In this section, we state the main result of the present work, namely the multisummability properties of the formal solutions to certain families of singularly perturbed moment differential equations.

Let $\mathbb{M}=(M_p)_{p\ge0}$ be a strongly regular sequence which admits a nonzero proximate order. 

Let $k, p\in\N$ with $1\le k<p$. Let $m_1=(m_1(p))_{p\ge 0}$ and $m_2=(m_2(p))_{p\ge 0}$ be two sequences of moments associated with two strongly regular sequences admitting nonzero proximate order, and assume that $m_1=(m_1(p))_{p\ge0}$ is the moment sequence associated with $\mathbb{M}^{s_1}$ for some $s_1>0$. Moreover, we assume that $m_2=(m_2(p))_{p\ge0}$ is an $\mathbb{M}$-sequence of order $s_2>0$, i.e., there exist positive constants $\tilde{c}_1,\tilde{c}_2$ such that
$$\tilde{c}_{1}^{p}(M_p)^{s_2}\le m_2(p)\le \tilde{c}_{2}^{p}(M_p)^{s_2},\qquad p\ge0.$$
We also assume that $s_2p>s_1k$, and that $\omega(\mathbb{M})\frac{s_2p}{k}<2$.

 Let $a(z)$ be a holomorphic function on some closed disc centered at the origin, say $\overline{D}$, such that $a(z)^{-1}\in\mathcal{O}(\overline{D})$. Moreover, $\hat{f}(z,\varepsilon)\in\C[[z,\varepsilon]]$ and $\hat{\psi}_j(\varepsilon)\in\C[[\varepsilon]]$ are formal power series.

We consider a singularly perturbed moment differential equation of the form 
\begin{equation}\label{eq:main}
 \left\{
 \begin{aligned}
 \varepsilon^k a(z)\partial_{m_2,z}^p\omega(z,\varepsilon)-\omega(z,\varepsilon)&=\hat{f}(z,\varepsilon)\\
  \partial_{m_2,z}^j\omega(0,\varepsilon)&=\hat{\psi}_j(\varepsilon),\qquad j=0,\ldots, p-1,
 \end{aligned}
 \right.
\end{equation}
where $\varepsilon$ is a small complex parameter. 

The main result of the present work reads as follows.

\begin{theo}\label{teopral}
\begin{itemize}
\item[(i)] There exists a unique formal solution $\hat{\omega}(z,\varepsilon)\in\mathbb{C}[[z,\varepsilon]]$ of (\ref{eq:main}), which belongs to $\mathcal{O}(\overline{D})[[\varepsilon]]$ if $\hat{f}\in\mathcal{O}(\overline{D})[[\varepsilon]]$. 
\item[(ii)] Let $s_1\in(0,\frac{s_2p}{k})$ and choose $(d_1,d_2)\in\R^2$ with 
$$|d_1-d_2|<\frac{\pi\omega(\mathbb{M})}{2}\left(\frac{s_2p}{k}-s_1\right).$$
The following statements are equivalent:
\begin{itemize}
\item[(ii.1)] $\hat{\omega}(z,\varepsilon)$ is $(\mathbb{M}^{s_1},\mathbb{M}^{\frac{s_2 p}{k}})$-summable in the multidirection $(d_1,d_2)\in\R^2$.
\item[(ii.2)]  $\hat{f}(z,\varepsilon)$ and $\partial_{m_2,z}^j\hat{\omega}(0,\varepsilon)$, $j=0,1,\ldots,p-1$, are $(\mathbb{M}^{s_1},\mathbb{M}^{\frac{s_2 p}{k}})$-summable in the multidirection $(d_1,d_2)$.
\end{itemize}
\end{itemize}
\end{theo}

The proof of the main result is left to the end of the work, and is preceded by some auxiliary results modifying the shape of the main problem or stating asymptotic results of related problems.

Let $\hat{u}=\varepsilon^k\hat{w}$. Then equation \eqref{eq:main} can be rewritten in the form
\begin{equation}\label{eq:2}
 a(z)\partial_{m_2,z}^p \hat{u}(z,\varepsilon)-\varepsilon^{-k}\hat{u}(z,\varepsilon)=\hat{f}(z,\varepsilon).
\end{equation}
Let us write $\hat{u}(z,\varepsilon)=\sum_{n\ge k} u_n(z)\frac{\varepsilon^n}{m_1(n)}$. Then the following lemma holds for the formal operator $\hat{\Bo}_{m_1,\varepsilon}$:

\begin{lemma}\label{lemma_1}
 If $\hat{u}\in\varepsilon^k \C[[z,\varepsilon]]$ then $\hat{\Bo}_{m_1,\varepsilon}(\varepsilon^{-k}\hat{u})=\partial_{m_1,\varepsilon}^k\hat{\Bo}_{m_1,\varepsilon}\hat{u}$.
\end{lemma}

\begin{proof}
Notice that
\begin{multline*}
\hat{\Bo}_{m_1,\varepsilon}(\varepsilon^{-k}\hat{u}(z,\varepsilon))=\sum_{n\ge 0}u_{n+k}(z)\frac{\varepsilon^n}{m_1(n)m_1(n+k)}\\=\partial_{m_1,\varepsilon}^k\sum_{n\ge 0}u_{n+k}(z)\frac{\varepsilon^{n+k}}{(m_1(n+k))^2}=\partial_{m_1,\varepsilon}^k\hat{\Bo}_{m_1,\varepsilon}\hat{u}(z,\varepsilon).
\end{multline*} 
\end{proof}

After applying the formal Borel transform $\hat{\Bo}_{m_1,\varepsilon}$ to both sides of \eqref{eq:2}, from Lemma \ref{lemma_1} we receive
\begin{equation}\label{eq:3}
 a(z)\partial_{m_2,z}^p \hat{U}(z,\varepsilon)-\partial_{m_1,\varepsilon}^k \hat{U}(z,\varepsilon)=\hat{F}(z,\varepsilon),
\end{equation}
where $\hat{U}(z,\varepsilon)=\hat{\Bo}_{m_1,\varepsilon}\hat{u}(z,\varepsilon)$ and $\hat{F}(z,\varepsilon)=\hat{\Bo}_{m_1,\varepsilon}\hat{f}(z,\varepsilon)$.

We fix $s_1,s_2>0$ such that $s_2p>s_1k$. Let $\mathbb{M}=(M_p)_{p\ge0}$ be a strongly regular sequence which admits nonzero proximate order. 

The following result is a direct consequence of Lemma~\ref{lema296}
\begin{lemma}\label{lema526}
The formal power series $\hat{\omega}(z,\varepsilon)$ is $(\mathbb{M}^{s_1},\mathbb{M}^{\frac{s_2p}{k}})$-multisummable (with respect to $\varepsilon$) in the multidirection $(d_1,d_2)$ if and only if  $\hat{u}(z,\varepsilon)$ is $(\mathbb{M}^{s_1},\mathbb{M}^{\frac{s_2p}{k}})$-multisummable (with respect to $\varepsilon$) in the multidirection $(d_1,d_2)$. Let $\omega$ and $u$ be the corresponding $(\mathbb{M}^{s_1},\mathbb{M}^{\frac{s_2p}{k}})$-sums. Then, it holds that $u=\varepsilon^k \omega$.
\end{lemma}

The next result is a direct consequence of Lemma~\ref{lema526} and Definition~\ref{defi257}.

\begin{prop}\label{prop00}
 The formal power series $\hat{w}(z,\varepsilon)$ is $(\mathbb{M}^{s_1},\mathbb{M}^{\frac{s_2 p}{k}})$-multisummable in the admissible multidirection $(d_1,d_2)$ if and only if the following conditions are met
 \begin{enumerate}
  \item $\hat{U}$ is $\mathbb{M}^{\frac{s_2 p}{k}-s_1}$-summable in the direction $d_2$,
  \item the sum of $\hat{U}$, denoted by $U$, is analytically continued to $S_{d_1}$ and $U\in\Oo^{\mathbb{M}^{s_1}}(S_{d_1},\EE)$.
 \end{enumerate}
\end{prop}

By Theorem 4 from \cite{LMS2} the formal power series $\hat{U}(z,\varepsilon)$ is $\mathbb{M}^{\frac{s_2 p}{k}-s_1}$-summable in the direction $d_2$ if and only if $\hat{F}(z,\varepsilon)$ and $\partial_{m_2,z}^j \hat{w}(0,\varepsilon)$, $j=0,1,\ldots,k-1$, are $\mathbb{M}^{\frac{s_2 p}{k}-s_1}$-summable in the same direction $d_2$.



The next result extends Lemma 4 and Theorem 4,~\cite{LMS2}. 

\begin{theo}[Compare Lemma 4 and Theorem 4, \cite{LMS2}]\label{teoaux} 
Let us consider the Cauchy problem 
\begin{equation}\label{e617}
 \left\{
 \begin{aligned}
  (\partial_{m_1,t}^k-a(z)\partial_{m_2,z}^p)u(t,z)&=\hat{f}(t,z)\in\C[[t,z]]\\
  \partial_{m_1,t}^ju(0,z)&=\varphi_j\in\Oo(\overline{D}),
 \end{aligned}
 \right.
\end{equation}
with $D$ being a fixed neighborhood of the origin.
\begin{itemize}
\item[1.] There exists a unique formal solution $\hat{u}(t,z)\in\mathbb{C}[[t,z]]$ of (\ref{e617}), which belongs to $\mathcal{O}(\overline{D})[[t]]$ if $\hat{f}\in\mathcal{O}(\overline{D})[[t]]$.
\item[2.] Let $\mathbb{E}$ denote the Banach space of holomorphic functions in $\overline{D}$ with the norm of the supremum. Assume that $\hat{f}\in\mathcal{O}(\overline{D})[[t]]$. The following statements are equivalent:
\begin{enumerate}
 \item[2.1.] $\hat{u}(t,z)$ is $\mathbb{M}^{\frac{s_2 p}{k}-s_1}$-summable in direction $d_2$ (seen as a formal power series in $t$ with coefficients in $\mathbb{E}$) with sum $u(t,z)$ being an analytic solution of (\ref{e617}), and moreover $u(t,z)\in\Oo^{\mathbb{M}^{s_1}}(S_{d_1},\EE')$, where $\EE'$ stands for the Banach space of holomorphic and bounded functions defined on $D(0,r')$ for some $0<r'<r$, endowed with the supremum norm.
 \item[2.2.] $\hat{f}(t,z)$ and $\partial_{m_2,z}^j\hat{u}(t,0)$ are $\mathbb{M}^{\frac{s_2 p}{k}-s_1}$-summable in direction $d_2$ with sums $f(t,z)$ and $\partial_{m_2,z}^j u(t,0)$, respectively. Moreover, $f(t,z)\in\Oo^{\mathbb{M}^{s_1}}(S_{d_1},\mathbb{E})$ and $\partial_{m_2,z}^ju(t,0)\in\Oo^{\mathbb{M}^{s_1}}(S_{d_1})$.
\end{enumerate}
\end{itemize}
\end{theo}

\begin{proof}
The proof heavily rests on that of Lemma 4 and Theorem 4,~\cite{LMS2}. We only give details at the points in which it differs from the proof of those previous results. 

The existence of a unique formal solution $\hat{u}(t,z)$ follows from the recursion satisfied by their coefficients, written as a formal power series  in $t$. Holomorphy of the coefficients is also guaranteed from that recursion formula.

For the second statement, we first observe that the implication ($2.1.\Rightarrow 2.2.)$ follows from the fact that $\mathbb{M}^{s_2p/k-s_1}$-summable formal power series along any fixed direction are compatible with respect to sums, product, and also moment derivation (see Corollary 1,~\cite{LMS2}), we obtain that $\hat{f}(t,z)$ and $\partial_{m_2,z}^{j}\hat{u}(t,0)$ are $\mathbb{M}^{s_2p/k-s_1}$-summable along the same direction. In addition to this, $f(t,z)\in\mathcal{O}^{\mathbb{M}^{s_1}}(S_{d_1},\mathbb{E})$ due to Theorem~\ref{teo2}. It is easy to check that $\partial_{m_2,z}^{j}u(t,0)\in\mathcal{O}^{\mathbb{M}^{s_1}}(S_{d_1})$.

We proceed to prove the implication ($2.2.\Rightarrow 2.1.)$.  Let $D:=D(0,r)$. We denote by $(\mathbb{E},\left\|\cdot\right\|_{\mathbb{E}})$ the Banach space of holomorphic and bounded functions on $\overline{D}$, where $\left\|\cdot\right\|_{\mathbb{E}}$ is the norm defined by 
$$\left\|f(z)\right\|_{r}:=\sum_{p\ge0}|f_p|r^p,$$
for $f\in\mathcal{O}(\overline{D})$, with $f(z)=\sum_{p\ge0}f_pz^p$ for $z\in\overline{D}$. Denoting $\hat{\omega}(t,z):=\partial_{m_2,z}^{p}\hat{u}(t,z)$, one has that $\hat{\omega}(t,z)$ is a formal solution of 
$$\left(1-\frac{1}{a(z)}\partial_{m_1,t}^{k}\partial_{m_2,z}^{-p}\right)\hat{\omega}(t,z)=\hat{g}(t,z),$$
where $\hat{g}(t,z)=a(z)^{-1}\partial_{m_1,t}^{k}(\hat{\psi}_0(t)+z\hat{\psi}_1(t)+\ldots+z^{p-1}\hat{\psi}_{p-1}(t))-a(z)^{-1}\hat{f}(t,z)$, and with $\hat{\psi}_0(t)=\hat{u}(t,0)$, $\hat{\psi}_j(t)=\frac{m_2(0)}{m_2(j)}\partial_{m_2,z}^{j}\hat{u}(t,0)$. We also define
$$\hat{\omega}(t,z)=\sum_{q\ge0}\hat{\omega}_q(t,z),$$
with $\hat{\omega}_0(t,z)=\hat{g}(t,z)$ and $\hat{\omega}_{q}(t,z)=a(z)^{-1}\partial_{m_1,t}^{k}\partial_{m_2,z}^{-p}\hat{\omega}_{q-1}(t,z)$ for $q\ge1$. From the hypotheses made one has that $\hat{\omega}_0(t,z)\in\mathbb{E}[[t]]$ is $\mathbb{M}^{s_2p/k-s_1}$-summable in direction $d_2$, with sum $\omega_0(t,z)\in\mathcal{O}(G\times\overline{D})$, for some sectorial region $G$ of opening larger than $\pi(s_2p/k-s_1)\omega(\mathbb{M})$ and bisecting direction $d_2$. By Theorem 4, \cite{LMS2} and by Theorem~\ref{teo2} we obtain that for every $G'\prec G$ and every $S'\prec S$ there exist positive constants $C_4,C_5,C_6$ such that
$$\left\|\partial_{m_1,t}^{n}\omega_0(t,z)\right\|_{r}\le C_4C_5^nm_1(n)M_n^{\frac{s_2p}{k}-s_1}\exp\left(M^{s_1}(C_6|t|)\right)\le \tilde{C}_1\tilde{C}_2^nM_n^{\frac{s_2p}{k}}\exp\left(M^{s_1}(C_6|t|)\right),$$
for all $t\in G'\cup S'$ and $n\in\N_0$, and some $\tilde{C}_1,\tilde{C}_2>0$. An induction argument and the application of Lemma 5~\cite{LMS2} (see the proof of Theorem 4, \cite{LMS2}) yield that $\hat{\omega}_{q}(t,z)\in\mathbb{E}[[t]]$ is $\mathbb{M}^{s_2p/k-s_1}$-summable in direction $d_2$, and
$$\left\|\partial_{m_1,t}^{n}\omega_q(t,z)\right\|_{\tilde{r}}\le  \tilde{C}_1C^q\tilde{C}_2^{qk+n}C_5^nM_{qk+n}^{\frac{s_2p}{k}}\frac{|z|^{pq}}{m_2(pq)}\exp\left(M^{s_1}(C_6|t|)\right),$$
for all $t\in G'\cup S'$, $z\in \overline{D}$ with $\tilde{r}=|z|$ and with $C=\left\|a(z)^{-1}\right\|_{r}$. It is direct to check that 
$$\sum_{q\ge0}\left\|\partial_{m_1,t}^{n}\omega_q(t,z)\right\|_{\tilde{r}}\le \tilde{C}_1\tilde{C}_3^nM_n^{\frac{s_2p}{k}}\exp\left(M^{s_1}(C_6|t|)\right),$$
valid for some $\tilde{C}_3>0$ and all $t\in G'\prec G$, all $z\in D(0,r')$, for some $r'>0$. This entails that $\omega(t,z):=\sum_{q\ge0}\omega_q(t,z)$ is a holomorphic function on $(G\cup S_{d_1})\times D(0,r')$ which is the $\mathbb{M}^{\frac{s_2p}{k}}$-sum of $\hat{\omega}(t,z)=\sum_{q\ge0}\hat{\omega}_q(t,z)\in\mathbb{E}[[t]]$ along direction $d_2$. A direct application of Watson's lemma (Corollary 4.12,~\cite{sanz}) allows us to conclude that the sum of $\hat{u}(t,z)\in\mathbb{E}[[t]]$, denoted by $u(t,z)$ is an analytic solution of (\ref{e617}) with $u(t,z)\in\Oo^{\mathbb{M}^{s_1}}(S_{d_1},\EE')$, where $\EE'$ stands for the Banach space of holomorphic and bounded functions defined on $D(0,r')$.
\end{proof}

All the previous arguments allow us to conclude with the proof of the main result of the present work.

\begin{proof}[{Proof of Theorem~\ref{teopral}}]
It is straightforward to check that (\ref{eq:main}) admits a unique formal solution which is obtained from the unique solution $\hat{U}(t,z)$ of (\ref{e617}) obtained in Theorem~\ref{teoaux} by reversing the relations $\hat{u}=\varepsilon^k\hat{\omega}$ and $\hat{U}(z,\varepsilon)=\hat{\mathcal{B}}_{m_1,\varepsilon}\hat{u}(z,\varepsilon)$.

If $\hat{\omega}(z,\varepsilon)\in\mathbb{E}[[\varepsilon]]$ is $(\mathbb{M}^{s_1},\mathbb{M}^{\frac{s_2p}{k}})$-summable in $(d_1,d_2)$, it is clear that $a(z)\partial_{m_2,z}^{p}\hat{\omega}(z,\varepsilon)\in\mathbb{E}[[\varepsilon]]$ is also $(\mathbb{M}^{s_1},\mathbb{M}^{\frac{s_2p}{k}})$-summable in $(d_1,d_2)$, and from Lemma~\ref{lema296} and Lemma~\ref{lema297} we have that $\varepsilon^ka(z)\partial_{m_2,z}^{p}\hat{\omega}(z,\varepsilon)$ also belongs to $\mathbb{E}\{\varepsilon\}_{(\mathbb{M}^{s_1},\mathbb{M}^{\frac{s_2p}{k}}),(d_1,d_2)}$. The sum of two elements in $\mathbb{E}\{\varepsilon\}_{(\mathbb{M}^{s_1},\mathbb{M}^{\frac{s_2p}{k}}),(d_1,d_2)}$ remains in that space, which in view of (\ref{eq:main}) entails that 
$$\hat{f}(z,\varepsilon)\in \mathbb{E}\{\varepsilon\}_{(\mathbb{M}^{s_1},\mathbb{M}^{\frac{s_2p}{k}}),(d_1,d_2)}.$$ 

It is clear that $\hat{\omega}\in \mathbb{E}\{\varepsilon\}_{(\mathbb{M}^{s_1},\mathbb{M}^{\frac{s_2p}{k}}),(d_1,d_2)}$ implies $\partial_{m_2,z}^{j}\hat{\omega}(0,\varepsilon)\in\mathbb{C}\{\varepsilon\}_{(\mathbb{M}^{s_1},\mathbb{M}^{\frac{s_2p}{k}}),(d_1,d_2)}$.

We proceed to give a proof for the implication $(ii.2)\Rightarrow (ii.1)$. In view of Proposition~\ref{prop00} one only has to check that $\hat{U}\in\mathbb{E}[[\varepsilon]]$ is $\mathbb{M}^{\frac{s_2p}{k}-s_1}$-summable in direction $d_2$ and its sum can be extended to an infinite sector of bisecting direction $d_1$, say $S_{d_1}$, being that extension in the space $\mathcal{O}^{\mathbb{M}^{s_1}}(S_{d_1},\mathbb{E})$. $\hat{U}$ turns out to be a formal solution of (\ref{eq:3}).

We observe that $\hat{F}(z,\varepsilon)\in\mathbb{E}[[\varepsilon]]$ and $\partial_{m_1,\varepsilon}^{j}\hat{U}(z,0)$ are both $\mathbb{M}^{\frac{s_2p}{k}-s_1}$-summable in direction $d_2$. On the one hand, $\hat{F}(z,\varepsilon)=\hat{\mathcal{B}}\hat{f}(z,\varepsilon)$ and $\hat{f}\in\mathbb{E}[[\varepsilon]]$ is $(\mathbb{M}^{s_1},\mathbb{M}^{\frac{s_2p}{k}})$-multisummable in $(d_1,d_2)$. On the other hand, we have
$$\hat{\mathcal{B}}_{m_1,\varepsilon}(\partial_{m_2,z}^{j}\hat{\omega}(0,\varepsilon))= \partial_{m_2,z}^{j}(\hat{\mathcal{B}}_{m_1,\varepsilon}\hat{\omega}(0,\varepsilon))=\partial_{m_2,z}^{j}(\partial_{m_1,\varepsilon}^{k}\hat{U})(0,\varepsilon),$$
which is $\mathbb{M}^{\frac{s_2p}{k}-s_1}$-summable in $d_2$. 

We are now in conditions to apply Theorem~\ref{teoaux} to arrive at $\hat{u}(z,\varepsilon)\in\mathbb{E}[[\varepsilon]]$ being $(\mathbb{M}^{s_1},\mathbb{M}^{\frac{s_2p}{k}})$-multisummable in $(d_1,d_2)$. We conclude after applying Lemma~\ref{lema296} that $\hat{\omega}(z,\varepsilon)\in\mathbb{E}[[\varepsilon]]$ is $(\mathbb{M}^{s_1},\mathbb{M}^{\frac{s_2p}{k}})$-multisummable in $(d_1,d_2)$.
\end{proof}

According to our best knowledge, the main result of the paper is new even in the case of singularly perturbed differential equations and classical multisummability (for the classical approach to multisummability see Section 10~\cite{balser} or Section 7~\cite{loday}). Namely, putting $\mathbb{M}=m_2=(p!)_{p\ge 0}$ and $s_2=1$ in Theorem \ref{teopral} we conclude that
\begin{corol}
\begin{itemize}
\item[(i)] Let $a(z)\in\Oo(\overline{D})$ be such that also $a(z)^{-1}\in\mathcal{O}(\overline{D})$. Moreover, $\hat{f}(z,\varepsilon)\in\C[[z,\varepsilon]]$ and $\hat{\psi}_j(\varepsilon)\in\C[[\varepsilon]]$ are formal power series. Then there exists a unique formal solution $\hat{\omega}(z,\varepsilon)\in\mathbb{C}[[z,\varepsilon]]$ of the singularly perturbed differential equation of the form
$$\left\{
 \begin{aligned}
 \varepsilon^k a(z)\partial_{z}^p\omega(z,\varepsilon)-\omega(z,\varepsilon)&=\hat{f}(z,\varepsilon)\\
  \partial_{z}^j\omega(0,\varepsilon)&=\hat{\psi}_j(\varepsilon),\quad j=0,\ldots, p-1,
 \end{aligned}
 \right.
$$
where $\varepsilon$ is a small complex parameter. If additionally $\hat{f}\in\mathcal{O}(\overline{D})[[\varepsilon]]$ 
then also $\hat{\omega}(z,\varepsilon)\in\mathcal{O}(\overline{D})[[\varepsilon]]$.
\item[(ii)] Let $s\in(0,\frac{p}{k})$ and choose $(d_1,d_2)\in\R^2$ with 
$|d_1-d_2|<\frac{\pi}{2}(\frac{p}{k}-s)$.
The following statements are equivalent:
\begin{itemize}
\item[(ii.1)] $\hat{\omega}(z,\varepsilon)$ is $(\frac{k}{p},\frac{1}{s})$-summable in the multidirection $(d_1,d_2)$.
\item[(ii.2)]  $\hat{f}(z,\varepsilon)$ and $\partial_{z}^j\hat{\omega}(0,\varepsilon)$, $j=0,1,\ldots,p-1$, are $(\frac{k}{p},\frac{1}{s})$-summable in the multidirection $(d_1,d_2)$.
\end{itemize}
\end{itemize}
\end{corol}


\begin{thebibliography}{99}
\bibitem{balser} W. Balser, \emph{Formal power series and linear systems of meromorphic ordinary differential equations}, Universitext, Springer-Verlag, New York, 2000.

\bibitem{BY} W. Balser, M. Yoshino, \textit{Gevrey order of formal power series solutions of inhomogeneous partial differential equations with constant coefficients}, Funkcial. Ekvac. 53 (2010), 411--434.

\bibitem{gomoyunov} M.I. Gomoyunov, \emph{On representation formulas for solutions of linear
differential equations with Caputo fractional derivatives}, Fract. Calc. Appl. Anal. 23, No 4 (2020), 1141--1160.

\bibitem{i} G.K. Immink, \emph{Exact asymptotics of nonlinear difference equations with levels 1 and 1+}, Ann. Fac. Sci. Toulouse T.XVII (2) (2008), 309--356.

\bibitem{i2} G. K. Immink, \emph{Accelero-summation of the formal solutions of nonlinear difference equations}, Ann. Inst. Fourier (Grenoble) 61 (2011), No 1, 1--51.

\bibitem{jkls} J. Jim\'enez-Garrido, S. Kamimoto, A. Lastra, J. Sanz, \textit{Multisummability in Carleman ultraholomorphic classes by means of nonzero proximate orders}, J. Math. Anal. Appl. 472 (2019), No. 1, 627--686. 

\bibitem{jss} J. Jim\'enez-Garrido, J. Sanz, G. Schindl, \emph{Injectivity and surjectivity of the
asymptotic Borel map in Carleman ultraholomorphic classes}, J. Math. Anal. Appl. 469 (2019), 136--168.

\bibitem{KST} A. Kilbas, H. Srivastava, J. Trujillo, \emph{Theory and applications of fractional differential equations}, North-Holland Math. Stud. 204, Elsevier, Amsterdam, 2006.

\bibitem{lama} A. Lastra, S. Malek, \emph{On multiscale Gevrey and $q-$Gevrey asymptotics for some linear $q-$difference$-$differential initial value Cauchy problems}, J. Difference Equ. Appl. 23 (2017), No. 8, 1397--1457. 

\bibitem{lastramaleksanz} A. Lastra, S. Malek, J. Sanz, \emph{Summability in general Carleman ultraholomorphic classes}, J. Math. Anal. Appl. 430 (2015), 1175--1206.

\bibitem{LMS0} A. Lastra, S. Michalik, M. Suwi\'nska, \emph{Estimates of formal solutions for some generalized moment partial differential equations}, J. Math. Anal. Appl. 500 (2021), no. 1, Paper No. 125094, 18 pp.

\bibitem{LMS} A. Lastra, S. Michalik, M. Suwi\'nska, \emph{Summability of formal solutions for a family of generalized moment integro-differential equations}, Fract. Calc. Appl. Anal. 24 (2021), no. 5,
1445--1476.

\bibitem{LMS2} A. Lastra, S. Michalik, M. Suwi\'nska, \emph{Summability of formal solutions for some generalized moment partial differential equations}, Results Math. 76, 22 (2021), 27 pp.

\bibitem{loday} M. Loday-Richaud, \emph{Divergent series, summability and resurgence II, Simple and multiple summability}, Lecture Notes in Math. 2154. Springer, 2016.

\bibitem{malek} S. Malek, \emph{Asymptotics and confluence for some linear $q-$difference$–$differential Cauchy problem}, J. Geom. Anal 32, 93 (2022). 

\bibitem{michalik13} S. Michalik, \emph{Analytic solutions of moment partial differential equations with constant coefficients}, Funkcial. Ekvac. 56 (2013), 19--50.

\bibitem{m2} S. Michalik, \emph{Summability of formal solutions of linear partial differential equations with divergent initial data}, J. Math. Anal. Appl. 406 (2013), no. 1, 243--260.

\bibitem{mt} S. Michalik, B. Tkacz, \emph{The Stokes phenomenon for some moment partial differential equations}, J. Dyn. Control Syst. 25 (2019), no. 4, 573--598.

\bibitem{re} P. Remy, \emph{Summability of the formal power series solutions of a certain class of inhomogeneous nonlinear partial differential equations with a single level}, J. Differ. Equations 313 (2022), 450--502.

\bibitem{re2} P. Remy, \emph{Gevrey order and summability of formal series solutions of certain classes of inhomogeneous linear integro-differential equations with variable coefficients}, J. Dyn. Control Syst. 23 (2017), 853--878.

\bibitem{RWF} L. Ren, J. Wang, M. Fe\v{c}kan, \emph{Asymptotically periodic solutions for Caputo type fractional evolution equations}, Fract. Calc. Appl. Anal. 21, No 5 (2019), 1294--1312.

\bibitem{sanz} J. Sanz, \emph{Flat functions in Carleman ultraholomorphic classes via proximate orders}, J. Math. Anal. Appl. 415 (2014), no. 2, 623--643.

\bibitem{sanzproceedings}  J. Sanz, \emph{Asymptotic analysis and summability of formal power series}, Analytic, Algebraic and Geometric Aspects of Differential Equations, Trends Math., Birkh\"auser/Springer, Cham (2017), 199--262. 

\bibitem{su} M. Suwi\'nska, \emph{Gevrey estimates of formal solutions for certain moment partial differential equations with variable coefficients}, J. Dyn. Control Syst. 27 (2021), no. 2, 355--370.

\bibitem{thilliez} V. Thilliez, \emph{Division by flat ultradifferentiable functions and sectorial extensions}, Results Math. 44 (2003), 169--188.

\end{thebibliography}
\end{document}